\documentclass{article}

     \usepackage[final]{neurips_2024}

\usepackage[utf8]{inputenc} 
\usepackage[T1]{fontenc}    
\usepackage{hyperref}       
\usepackage{url}            
\usepackage{booktabs}       
\usepackage{amsfonts}       
\usepackage{nicefrac}       
\usepackage{microtype}      
\usepackage{xcolor}         

\usepackage{natbib}
\usepackage{amssymb}
\usepackage{mathtools}
\usepackage{fancybox}
\usepackage{multirow}
\usepackage{paralist}
\usepackage{subfigure}
\usepackage{booktabs}
\RequirePackage{algorithm}
\RequirePackage{algorithmic}
\newtheorem{theorem}{Theorem}

\newtheorem{proof}{Proof}
\newtheorem{lemma}{Lemma}

\newtheorem{ass}{Assumption}

\newcommand{\Norm}[1]{\left\|#1\right\|}
\usepackage{makecell} 
\def \cite {\citep}
\def \E {\mathbb{E}}
\def \R {\mathbb{R}}

\def \z {\mathbf{z}}
\def \u {\mathbf{u}}
\def \v {\mathbf{v}}

\def \x {\mathbf{x}}
\def \y {\mathbf{y}}

\def \H {\mathcal{H}}
\def \LL {L}

\usepackage{pifont}
\newcommand{\rA}{\mathrm{(A)}}
\newcommand{\rB}{\mathrm{(B)}}
\newcommand{\rC}{\mathrm{(C)}}

\newcommand{\cmark}{\text{\ding{51}}}
\newcommand{\xmark}{\text{\ding{55}}}

\title{Adaptive Variance Reduction for Stochastic Optimization under Weaker Assumptions}

\author{
  Wei Jiang\textsuperscript{\rm 1},~~Sifan Yang\textsuperscript{\rm 1,2},~~Yibo Wang\textsuperscript{\rm 1,2},~~Lijun Zhang\textsuperscript{\rm 1,2,}\thanks{Lijun Zhang is the corresponding author.}\\
  \textsuperscript{\rm 1}National Key Laboratory for Novel Software Technology, Nanjing University, Nanjing, China \\
  \textsuperscript{\rm 2}School of Artificial Intelligence, Nanjing University, Nanjing, China \\
\texttt{\{jiangw,yangsf,wangyb,zhanglj\}@lamda.nju.edu.cn}
 }

\begin{document}

\maketitle

\begin{abstract}
This paper explores adaptive variance reduction methods for stochastic optimization based on the STORM technique. Existing adaptive extensions of STORM rely on strong assumptions like bounded gradients and bounded function values, or suffer an additional $\mathcal{O}(\log T)$ term in the convergence rate. To address these limitations, we introduce a novel adaptive STORM method that achieves an optimal convergence rate of $\mathcal{O}(T^{-1/3})$ for non-convex functions with our newly designed learning rate strategy. Compared with existing approaches, our method requires weaker assumptions and attains the optimal convergence rate without the additional $\mathcal{O}(\log T)$ term. We also extend the proposed technique to stochastic compositional optimization, obtaining the same optimal rate of $\mathcal{O}(T^{-1/3})$. Furthermore, we investigate the non-convex finite-sum problem and develop another innovative adaptive variance reduction method that achieves an optimal convergence rate of $\mathcal{O}(n^{1/4} T^{-1/2} )$, where $n$ represents the number of component functions. Numerical experiments across various tasks validate the effectiveness of our method.
\end{abstract}

\section{Introduction}
This paper investigates the stochastic optimization problem
\begin{align}\label{prob}
    \min_{\x \in \R^d} f(\x),
\end{align}
where $f:\R^{d}\mapsto \R$ is a smooth non-convex function. We assume that only noisy estimations of its gradient $\nabla f(\x)$ can be accessed, denoted as $\nabla f(\x;\xi)$, where $\xi$ represents the random sample drawn from a stochastic oracle such that $\E[ \nabla f(\x;\xi) ] = \nabla f(\x)$. 

Problem~(\ref{prob}) has been comprehensively investigated in the literature~\cite{JMLR:v12:duchi11a,kingma:adam,loshchilov2017sgdr}, and it is well-known that the classical stochastic gradient descent~(SGD) achieves a convergence rate of $\mathcal{O}(T^{-1/4})$, where $T$ denotes the iteration number~\cite{SGD}. To further improve the convergence rate, variance reduction methods have been developed, and attain an improved rate of $\mathcal{O}(T^{-1/3})$ under a slightly stronger smoothness assumption~\cite{Fang2018SPIDERNN,Wang2018SpiderBoostAC}. However, these methods necessitate the use of a huge batch size in each iteration, which is often impractical to use. To eliminate the need for large batches, a momentum-based variance reduction method --- STORM~\cite{cutkosky2019momentum} is introduced, which achieves a convergence rate of $\mathcal{O}(T^{-1/3}\log T )$.

Although aforementioned methods are equipped with convergence guarantees, their analyses rely on delicate configurations of hyper-parameters, such as the learning rate and the momentum parameter. To set them properly, the algorithm typically needs to know the value of the smoothness parameter $L$, the gradient upper bound $G$, and the variance upper bound $\sigma$, which are often unknown in practice. Specifically, most algorithms require the learning rate $\eta_t$ smaller than $\mathcal{O}(1/L)$, and for the STORM method, setting the momentum parameter to $\mathcal{O}(L^2 \eta_t^2)$ is crucial for ensuring convergence~\cite{cutkosky2019momentum}.  

To overcome this limitation, many adaptive algorithms have been developed, aiming to obtain convergence guarantees without prior knowledge of problem-dependent parameters such as $L$, $G$ and $\sigma$. Based on the STORM method, \citet{levy2021storm} develop the STORM+ algorithm, which attains the optimal $\mathcal{O}(T^{-1/3})$ convergence rate under the assumption of bounded function values and gradients. To remove the need for the bounded function values assumption, \citet{Liu2022METASTORMGF} propose the META-STORM algorithm to attain an $\mathcal{O}(T^{-1/3} \log T)$ convergence rate, but it still requires the bounded gradients assumption and includes an additional $\mathcal{O}(\log T)$ term. In summary, despite advancements in this field, existing adaptive STORM-based methods either depend on strong assumptions or suffer an extra $\mathcal{O}(\log T)$ term compared with the lower bound~\citep{Arjevani2019LowerBF}. Hence, a fundamental question to be addressed is:

\shadowbox{\begin{minipage}[t]{0.95\columnwidth}%
\it Is it possible to develop an adaptive STORM method that achieves the optimal convergence rate for non-convex functions under weaker assumptions?
\end{minipage}} 

We give an affirmative answer to the above question by devising a novel optimal Adaptive STORM method~(Ada-STORM). The learning rate of our algorithm is set to be inversely proportional to a specific power of the iteration number $T$ in the initial iterations, and then changes adaptively based on the cumulative sum of past gradient estimations. In this way, we are able to adjust the learning rate dynamically according to the property of stochastic gradients, and ensure a small learning rate in the beginning. Leveraging this strategy, Ada-STORM achieves an optimal convergence rate of $\mathcal{O}(T^{-1/3})$ for non-convex functions.  Notably, our analysis does not require the function to have bounded values and bounded gradients, which is a significant advancement over existing methods~\cite{levy2021storm,Liu2022METASTORMGF}. Additionally, our convergence rate does not contain the extra $\mathcal{O}(\log T) $ term, which is often present in STORM-based methods~\cite{ cutkosky2019momentum,Liu2022METASTORMGF}. To highlight the versatility of our approach and its potential impact in the field of stochastic optimization, we further extend our technique to develop optimal adaptive methods for compositional optimization.

Finally, we investigate adaptive variance reduction for the non-convex finite-sum problems. Inspired by SAG algorithm~\citep{DBLP:conf/nips/RouxSB12}, we incorporate an additional term in the STORM estimator, which measures the difference of past gradients between the selected component function and the overall objective. By changing the learning rate according to the sum of past gradient estimations, we are able to obtain an optimal convergence rate of $\mathcal{O}(n^{-1/4}T^{-1/2})$ for finite-sum problems in an adaptive manner, where $n$ is the number of component functions. Our result is better than the previous convergence rate of $\mathcal{O}(n^{-1/4}T^{-1/2} \log (nT))$ obtained by adaptive SPIDER method~\citep{NEURIPS2022_94f625dc}. In summary, compared with existing methods, this paper enjoys the following advantages:
\begin{itemize}
\item For stochastic non-convex optimization, our method achieves the optimal convergence rate of $\mathcal{O}(T^{-1/3})$ under more relaxed assumptions. Specifically, it does not require the bounded function values or the bounded gradients, and does not include the additional $\mathcal{O}(\log T)$ term in the convergence rate.

\item Our learning rate design and the analysis exhibit broad applicability. We substantiate this claim by obtaining an optimal rate of $\mathcal{O}(T^{-1/3})$ for stochastic compositional optimization, using the technique proposed in this paper.

\item For non-convex finite-sum optimization, we further improve our adaptive algorithm to attain an optimal convergence rate of $\mathcal{O}(n^{1/4}T^{-1/2})$, which outperforms the previous result by eliminating the $\mathcal{O}( \log (nT ) )$ factor.
\end{itemize}
A comparison between our method and other STORM-based algorithms is shown in Table~\ref{sample}. Numerical experiments on different tasks also validate the effectiveness of the proposed method.
\begin{table*}[t]
\caption{Summary of results for STORM-based methods. Here, NC denotes non-convex,  Comp. indicates compositional optimization, FS represents finite-sum optimization, and BG/BF refers to requiring bounded gradients or bounded function values assumptions. Adaptive means the method does not require to know problem-dependent parameters, i.e., $L$, $G$, and $\sigma$.}
\label{sample}
\begin{center}
\begin{small}
\resizebox{\textwidth}{!}{
\begin{tabular}{lcccc}
\toprule
Method & Setting &   Convergence Rate & Adaptive & BG/BF \\
\midrule
STORM~\cite{cutkosky2019momentum}  & NC &   $\mathcal{O}\left( T^{-1/3}\log T\right)$ & \xmark & \cmark \\
Super-ADAM~\cite{NEURIPS2021_4be5a36c}  & NC &   $\mathcal{O}\left( T^{-1/3}\log T\right)$ & \xmark & -- \\
STORM+~\cite{levy2021storm}    & NC & $\mathcal{O}\left(T^{-1/3}\right)$ & \cmark & \cmark \\
META-STORM~\cite{Liu2022METASTORMGF} & NC & $\mathcal{O}\left( T^{-1/3}\log T\right)$ & \cmark &  \cmark \\
\midrule
\textbf{Theorem~\ref{thm:main_0}, \ref{thm:main_1}}   & NC &  $\mathcal{O}\left(T^{-1/3}\right)$& \cmark & -- \\
\textbf{Theorem~\ref{thm:main_0+}}  & NC \& Comp. &  $\mathcal{O}\left(T^{-1/3} \right)$& \cmark & -- \\
\textbf{Theorem~\ref{thm:main_2}}  & NC \& FS &  $\mathcal{O}\left(n^{1/4}T^{-1/2} \right)$& \cmark & -- \\
\bottomrule
\end{tabular}}
\end{small}
\end{center}
\end{table*}
\section{Related work}
This section briefly reviews related work on stochastic variance reduction methods and adaptive stochastic algorithms.
\subsection{Stochastic variance reduction methods}
Variance reduction has been widely used in stochastic optimization to reduce the gradient estimation error and thus improve the convergence rates. The idea of variance reduction can be traced back to the SAG algorithm~\cite{DBLP:conf/nips/RouxSB12}, which incorporates a memory of previous gradient values to ensure variance reduction and achieves a linear convergence rate for strongly convex finite-sum optimization. To avoid the storage of past gradients, SVRG ~\citep{NIPS2013_ac1dd209,NIPS:2013:Zhang} proposes to calculate the full gradient periodically, obtaining the same convergence rate as the SAG algorithm. Subsequent advancement has been made by the SARAH method~\citep{arxiv.1703.00102}, which derives better convergence for smooth convex functions. 

In the context of non-convex objectives, \citet{Fang2018SPIDERNN} introduce the SPIDER estimator, which improves the convergence rate from $\mathcal{O}(T^{-1/4})$ to $\mathcal{O}(T^{-1/3})$ in stochastic settings, and to $O(n^{1/4}T^{-1/2})$ in finite-sum scenarios, with $n$ representing the number of components in the finite-sum. 
Following this, the SpiderBoost algorithm~\citep{Wang2018SpiderBoostAC}  refines the SPIDER approach by employing a larger constant step size and adapting it for composite optimization problems. However, a common limitation among these methods is their reliance on large batch sizes for each iteration, posing practical challenges due to high computational demands. To mitigate this issue, \citet{cutkosky2019momentum} introduce the STORM method, a momentum-based technique that achieves an $\mathcal{O}(T^{-1/3}\log T)$ convergence rate without using large batches. Concurrently, \citet{trandinh2019hybrid} obtain the same result using a similar algorithm but through a different analysis.
\subsection{Adaptive stochastic algorithms}
For stochastic optimization, it is well-known that the SGD algorithm can obtain a convergence rate of $\mathcal{O}(T^{-1/4})$ for non-convex objective functions with well-designed learning rates~\cite{SGD}. Instead of using pre-defined iteration-based learning rates, many stochastic methods propose to adjust the learning rate based on past stochastic gradients. One of the foundational works is the AdaGrad algorithm~\cite{JMLR:v12:duchi11a}, which proves to be effective for sparse data. Further advancements include RMSprop~\cite{rmsprop} and  Adam~\cite{kingma:adam}, demonstrating broad effectiveness across a wide range of machine learning problems. Later, the Super-Adam~\cite{NEURIPS2021_4be5a36c} algorithm further improves the Adam algorithm via the variance reduction technique STORM~\cite{cutkosky2019momentum} and obtains a convergence rate of $\mathcal{O}(T^{-1/3}\log T)$. Nevertheless, to obtain the corresponding convergence rates, these methods still require knowledge of certain problem-dependent parameters to set hyper-parameters accurately, hence not  adaptive.\footnote{In this paper, adaptive means the algorithm does not require problem-dependent parameters to set up hyper-parameters such as the learning rate and the momentum parameter.} To solve this problem, many research aims to develop fully adaptive SGD methods that maintain the optimal convergence rate without knowing problem-specific parameters~\cite{NIPS2014_0ff8033c, Chen_Langford_Orabona_2022,pmlr-v178-carmon22a, Ivgi2023DoGIS,yang2023two}. 

Recently, adaptive adaptations of STORM have received considerable attention. A notable development is the introduction of STORM+ \cite{levy2021storm}, which presents a fully adaptive version of STORM while attaining an optimal convergence rate. To circumvent the bounded function values assumption in STORM+, the META-STORM~\cite{Liu2022METASTORMGF} approach is developed, equipped with a nearly optimal bound. However, META-STORM still requires the bounded gradients assumption, and it includes an additional $\mathcal{O}(\log T)$ in the convergence rate. Consequently, adaptive STORM with the optimal convergence rate and under mild assumptions still needs further explorations.

\section{Adaptive variance reduction for non-convex optimization}
In this section, we develop an adaptive STORM method for non-convex functions. We first outline the assumptions used, and then present our proposed method and analyze its convergence rate.

\subsection{Assumptions}
We introduce the following assumptions, which are standard and commonly adopted in the stochastic optimization~\cite{arxiv.1703.00102,Fang2018SPIDERNN,cutkosky2019momentum,pmlr-v139-li21a}.

\begin{ass} (Average smoothness)\label{ass1}
\begin{equation*}
\begin{split}
\mathbb{E}\left[\left\|\nabla f(\x;\xi) -\nabla f(\y;\xi)\right\|^{2}\right] \leq L^2\|\mathbf{x}-\mathbf{y}\|^{2}.
\end{split}
\end{equation*} 
\end{ass}

\begin{ass}\label{ass2}  (Bounded variance)
\begin{equation*}
\begin{split}
 \mathbb{E}\left[\left\|\nabla f(\x;\xi) -\nabla f(\mathbf{x})\right\|^{2}\right] \leq \sigma^{2}.
\end{split}
\end{equation*} 
\end{ass}

\begin{ass}\label{ass3} $f_{*}=\inf_{\x} f(\x) \geq-\infty$ and $f\left(\x_{1}\right)-f_{*} \leq \Delta_{f}$ for the initial solution $\x_{1}$.
\end{ass}
Note that some additional assumptions are required in other STORM-based methods. Specifically, STORM~\cite{cutkosky2019momentum}, STORM+~\cite{levy2021storm}, and META-STORM~\cite{Liu2022METASTORMGF} assume the bounded gradients. Moreover, STORM+ makes an additional assumption of the bounded function values. 



\subsection{The proposed method}
In this subsection, we aim to develop an adaptive STORM method that achieves an optimal convergence rate for non-convex functions under weaker assumptions. Our algorithm framework is the same as the original STORM~\cite{cutkosky2019momentum}, and the only difference is the setup of the momentum parameter $\beta_t$ and the learning rate $\eta_t$. First, we present the STORM algorithm in Algorithm~\ref{alg:storm}.
\begin{algorithm}[t]
	\caption{STORM Algorithm}
	\label{alg:storm}
	\begin{algorithmic}[1]
	\STATE {\bfseries Input:} time step $T$, initial point $\x_1$
	\FOR{time step $t = 1$ {\bfseries to} $T$}
        \STATE Set hyper-parameters $\beta_t$ and $\eta_t$
        \STATE Compute $\v_t$ according to equation~(\ref{STORM-v})
		\STATE Update the decision variable: $\x_{t+1} = \x_t - \eta_t \v_t$
		\ENDFOR
	\STATE Choose $\tau$ uniformly at random from $\{1, \ldots, T\}$
	\STATE Return $\x_\tau$
	\end{algorithmic}
\end{algorithm}

The core idea of STORM lies in a carefully devised variance reduced estimator $\v_t$, which effectively tracks the gradient  $\nabla f(\x_t)$. For the first iteration~($t=1$), we set $\v_1=\sum_{i=1}^{B_0} \frac{1}{B_0}\nabla f(\x_1;\xi_1^i)$, which is estimated within a batch $B_0 = T^{1/3}$. Note that we use large batch only in the first iteration, and constant batch size in other iterations. In subsequent iterations~($t\geq2$), estimator $\v_t$ is defined as:
\begin{align}\label{STORM-v}
\begin{split}
        \v_t = (1-\beta_t) \v_{t-1} + \beta_t \nabla f(\x_t;\xi_t) +(1-\beta_t) \left(\nabla f(\x_{t};\xi_t)-\nabla f(\x_{t-1};\xi_t)\right), 
\end{split}
\end{align}
where the first two terms are similar to the momentum SGD, and the last term serves as the error correction, which ensures the variance reduction effect. By choosing the values of $\beta_t$ and $\eta_t$ carefully, STORM ensures that the estimation error $\E[ \Norm{\v_t - \nabla f(\x_t)}^2 ]$ would decrease gradually. In the original STORM paper, these parameters are set up as:
\begin{align*}
    \eta_t = \frac{k}{\left( w+\sum_{i=1}^t \Norm{\nabla f(\x_{t};\xi_t)}^2\right)^{1/3}}, \quad
    \beta_t = c \eta_t^2,
\end{align*}
where $k=\mathcal{O}(G^{2/3} L^{-1})$, $w=\mathcal{O}(G^2)$ and  $c=\mathcal{O}(L^2)$. The settings of these hyper-parameters are crucial to the convergence analysis of STORM. However, it is worth noting that $L$ is the smoothness parameter and $G$ is the gradient upper bound, which are often difficult to determine in practice.

To address this problem, our approach defines the hyper-parameters as follows:
\begin{align}\label{main:eq}
\begin{split}
    \eta_t= \min \left\{ \frac{1}{T^{1/3}}, \frac{1}{T^{(1-\alpha)/3} \left( \sum_{i=1}^{t} \Norm{\v_i}^2\right)^{\alpha}} \right\}, \quad
    \beta_t =\beta = T^{-2/3}, 
\end{split}
\end{align}
where $0<\alpha <1/3$. Notably, our method does not rely on the parameters $L$ and $G$, and also does not need the bounded gradients or bounded function values assumptions that are common in other methods. Although our formulation initially requires knowledge of the iteration number $T$, this can be effectively circumvented using the doubling trick, which will be explained later. 
The above learning rate $\eta_t$ can also be expressed in an alternative, more illustrative manner:
\begin{equation*}
\begin{split}
    \eta_t=\left\{\begin{array}{ll}
\frac{1}{T^{1/3}} &  \text{if} \ \sum_{i=1}^{t} \Norm{\v_i}^2 \leq {T^{1/3}};\\ \\
\frac{1}{T^{(1-\alpha)/3} \left( \sum_{i=1}^{t} \Norm{\v_i}^2\right)^{\alpha}}   &  \text{else.}
\end{array}\right. 
\end{split}
\end{equation*}
This formulation ensures that the learning rate starts sufficiently small in the initial stages and then changes dynamically based on the gradient estimator $\v_t$. This design makes our learning rate setup and convergence analysis distinctly different from previous methods. Next, we present the following theoretical guarantee for our  algorithm.

\begin{theorem}\label{thm:main_0} Under Assumptions \ref{ass1}, \ref{ass2} and \ref{ass3},   Algorithm~\ref{alg:storm} with hyper-parameters in equation~(\ref{main:eq}) guarantees that:
\begin{align*}
    \E\left[ \Norm{\nabla f(\x_{\tau})} \right] \leq \mathcal{O}\left(\frac{\Delta_f^{\frac{1}{2(1-\alpha)}}+\sigma^{\frac{1}{1-\alpha}} + L^{\frac{1}{2\alpha}}}{T^{1/3}}\right).
\end{align*}
\end{theorem}
\textbf{Remark:} To ensure that $\E[ \Norm{\nabla f(\x_{\tau})} ] \leq \epsilon$, the overall complexity is $\mathcal{O}(\epsilon^{-3})$, which is known to be optimal up to constant factors~\cite{Arjevani2019LowerBF}. Compared with existing STORM-based algorithms~\cite{cutkosky2019momentum,levy2021storm,Liu2022METASTORMGF}, our method does not have the extra $\mathcal{O}(\log T)$ term in the convergence rate, and our analysis does not require bounded gradients or bounded function values assumptions. Also note that the selection of $\alpha$ does not affect the order of $T$, and larger $\alpha$ leads to better dependence on parameter $L$ and worse reliance on parameters $\Delta$ and $\sigma$. Considering we require that $0<\alpha<1/3$, we can simply set $\alpha=0.3$ in practice.

\subsection{The doubling trick}
While we have attained the optimal convergence rate using the proposed adaptive STORM method, it requires knowing the total number of iterations $T$ in advance. Here, we show that we can avoid this requirement by using the doubling trick, which divides the algorithm into several stages and increases the iteration number in each stage gradually. Specifically, we design a multi-stage algorithm over $k=\{1,2,\cdots,K\}$ stages. At the beginning of each new stage, we reset $\x_t=\x_0$. In each stage $k$, the STORM algorithm is executed for $2^{k-1}$ iterations, effectively doubling the iteration numbers after each stage. In any step $t$, we first identify the current stage as $1+\lfloor \log t\rfloor$ and then calculate the iteration number for this stage as $I_t = 2^{\lfloor \log t\rfloor}$. Then, we can set the hyper-parameters as:
\begin{align}\label{db-nc}
\begin{split}
    \eta_t= \min \left\{ \frac{1}{I_{t}^{1/3}}, \frac{1}{I_{t}^{(1-\alpha)/3} \left( \sum_{i={I_t}}^{t} \Norm{\v_i}^2\right)^{\alpha}} \right\}, \quad
    \beta_t  = I_{t}^{-2/3}, \quad I_t = 2^{\lfloor \log t\rfloor}. 
\end{split}
\end{align}
This approach eliminates the need to predetermine the iteration number $T$. By using the doubling trick, we can still obtain the same optimal convergence rate as stated in the following theorem.

\begin{theorem}\label{thm:main_1} Under Assumptions \ref{ass1}, \ref{ass2} and \ref{ass3}, Algorithm~\ref{alg:storm} with hyper-parameters in equation~(\ref{db-nc}) guarantees that:
\begin{align*}
    \E \left[\Norm{\nabla f(\x_\tau)} \right]\leq \mathcal{O}\left(\frac{\Delta_f^{\frac{1}{2(1-\alpha)}}+\sigma^{\frac{1}{1-\alpha}} + L^{\frac{1}{2\alpha}}}{{T}^{1/3}}\right).
\end{align*}
\end{theorem}

\section{Extension to stochastic compositional optimization}
To demonstrate the broad applicability of our proposed  technique, we extend it to stochastic compositional optimization~\cite{wang2017stochastic,DBLP:journals/jmlr/WangLF17,Yuan2019EfficientSN,Zhang2019ASC,Zhang2021MultiLevelCS,ICML:2023:Jiang,ICML:2024:Jiang}, formulated as:
\begin{align}\label{comp}
    \min_{\x \in \R^d} F(\x) = f(g(\x)),
\end{align}
where $f$ and $g$ are smooth functions. We assume that we can only access to unbiased estimations of $\nabla f(\x)$, $\nabla g(\x)$ and $g(\x)$, denoted as $\nabla f(\x;\xi)$, $\nabla g(\x;\zeta)$ and $g(\x;\zeta)$. Here $\xi$ and $\zeta$ symbolize the random sample drawn for a stochastic oracle such that $\E[ \nabla f(\x;\xi) ]=\nabla f(\x)$, $\E[  g(\x;\zeta) ]=g(\x)$, and $\E[ \nabla g(\x;\zeta)]=\nabla g(\x)$. 

Existing variance reduction methods~\cite{NEURIPS2019_21ce6891,Zhang2019ASC,qi2021online} are able to obtain optimal $\mathcal{O}(T^{-1/3})$ convergence rates for problem~(\ref{comp}), but they require the knowledge of smoothness parameter and the gradient upper bound to set up hyper-parameters. In this section, we aim to achieve the same optimal convergence rate without prior knowledge of problem-dependent parameters. We develop our adaptive algorithm for this problem as follows. In each step $t$, the algorithm maintains an inner function estimator $\u_t$ in the style of STORM, i.e.,
\begin{align}\label{CP-u}
\begin{split}
    \u_t = (1-\beta) \u_{t-1} +  g(\x_t;\zeta_t)- (1-\beta) g(\x_{t-1};\zeta_t). 
\end{split}  
\end{align}
Then, we construct a gradient estimator $\v_t$ based on $\u_t$ also in the style of STORM:
\begin{align}\label{CP-v}
\begin{split}
    \v_t = (1-\beta) \v_{t-1} +  \nabla f(\u_t;\xi_t)\nabla g(\x_t;\zeta_t) -(1-\beta)  \nabla f(\u_{t-1};\xi_t)\nabla g(\x_{t-1};\zeta_t).
\end{split}    
\end{align}

After that, we apply gradient descent using the gradient estimator $\v_t$. The whole algorithm is presented in Algorithm~\ref{alg:storm_cp}, and hyper-parameters are set the same as in equation~(\ref{main:eq}). For the first iteration, we simply set $\u_1=\sum_{i=1}^{B_0} \frac{1}{B_0}g(\x_1;\zeta_1^i)$ and $\v_1=\sum_{i=1}^{B_0} \frac{1}{B_0}\nabla f(\u_1;\xi_1^i)\nabla g(\x_1;\zeta_1)$, where $B_0 = T^{1/3}$. Next, we list common assumptions used in the literature of compositional optimization~\cite{wang2017stochastic,DBLP:journals/jmlr/WangLF17,Yuan2019EfficientSN,Zhang2019ASC,Zhang2021MultiLevelCS}.
\begin{algorithm}[t]
	\caption{Compositional STORM}
	\label{alg:storm_cp}
	\begin{algorithmic}[1]
	\STATE {\bfseries Input:} time step $T$, initial point $\x_1$
	\FOR{time step $t = 1$ {\bfseries to} $T$}
        \STATE Compute $\u_t$ according to equation~(\ref{CP-u})
        \STATE Compute $\v_t$ according to equation~(\ref{CP-v})
		\STATE Update the decision variable: $\x_{t+1} = \x_t - \eta_t \v_t$
		\ENDFOR
	\STATE Choose $\tau$ uniformly at random from $\{1, \ldots, T\}$
	\STATE Return $\x_\tau$
	\end{algorithmic}
\end{algorithm}

\begin{ass} (Average smoothness and Lipschitz continuity)\label{ass4}
\begin{equation*}
\begin{split}
\mathbb{E}\left[\left\|\nabla f(\x;\xi) -\nabla f(\y;\xi)\right\|^{2}\right] \leq L\|\mathbf{x}-\mathbf{y}\|^{2}; \ \mathbb{E}\left[\left\|f(\x;\xi) -f(\y;\xi)\right\|^{2}\right] \leq C \|\mathbf{x}-\mathbf{y}\|^{2}; \\
\mathbb{E}\left[\left\|\nabla g(\x;\zeta) -\nabla g(\y;\zeta)\right\|^{2}\right] \leq L\|\mathbf{x}-\mathbf{y}\|^{2};\  
\mathbb{E}\left[\left\|g(\x;\zeta) -g(\y;\zeta)\right\|^{2}\right] \leq C \|\mathbf{x}-\mathbf{y}\|^{2}.
\end{split}
\end{equation*} 
\end{ass}

\begin{ass}\label{asm:stochastic2+}  (Bounded variance)
\begin{equation*}
\begin{split}
	\mathbb{E}\left[\left\|g(\x;\zeta) - g(\mathbf{x})\right\|^{2}\right] \leq \sigma^{2}; 
 \mathbb{E}\left[\left\|\nabla g(\x;\zeta) -\nabla g(\mathbf{x})\right\|^{2}\right] \leq \sigma^{2}; 
 \mathbb{E}\left[\left\|\nabla f(\x;\xi) -\nabla f(\mathbf{x})\right\|^{2}\right] \leq \sigma^{2}.
\end{split}
\end{equation*} 
\end{ass}

\begin{ass}\label{asm:stochastic4+} $F_{*}=\inf_{\x} F(\x) \geq-\infty$ and $F\left(\x_{1}\right)-F_{*} \leq \Delta_{F}$ for the initial solution $\x_{1}$.
\end{ass}
\textbf{Remark:} In Assumption~\ref{ass4}, we further require standard Lipschitz continuity assumption, which is essential and widely required in the literature for stochastic compositional optimization~\cite{DBLP:journals/jmlr/WangLF17, Yuan2019EfficientSN,jiang2022multiblocksingleprobe,jiang2022optimal}. This assumption is inherently introduced by the compositional optimization itself rather than by our adaptive techniques.

With the above assumptions, our algorithm enjoys the following guarantee. 
\begin{theorem}\label{thm:main_0+} Under Assumptions \ref{ass4}, \ref{asm:stochastic2+} and \ref{asm:stochastic4+}, our Algorithm~\ref{alg:storm_cp} ensures that:
\begin{align*}
    \E\left[ \Norm{\nabla F(\x_\tau)} \right]\leq \mathcal{O}\left(T^{-1/3}\right).
\end{align*}
\end{theorem}
\textbf{Remark:} This rate matches the state-of-the-art~(SOTA) results in stochastic compositional optimization~\cite{NEURIPS2019_21ce6891,Zhang2019ASC,qi2021online}, and our method achieve this in an adaptive manner. Note that our convergence rate aligns with the lower bound for single-level problems~\cite{Arjevani2019LowerBF} and is thus unimprovable.

\section{Adaptive variance reduction for  finite-sum optimization}
In this section, we further improve our adaptive variance reduction method to obtain an enhanced convergence rate for non-convex finite-sum optimization, which is in the form of
\begin{align*}
    \min_{\x \in \R^d} F(\x) = \frac{1}{n}\sum_{i=1}^{n} f_i(\x),
\end{align*}
where each $f_i(\cdot)$ is a smooth non-convex function. Existing adaptive method for this problem~\cite{NEURIPS2022_94f625dc} achieves a convergence rate of $\mathcal{O}(n^{1/4}T^{-1/2} \log (nT ) )$ based on the variance reduction technique SPIDER~\cite{Fang2018SPIDERNN}, suffering from an extra $\mathcal{O}(\log (nT))$ term compared with the corresponding lower bound~\citep{Fang2018SPIDERNN,pmlr-v139-li21a}.

To obtain the optimal convergence rate for finite-sum optimization, we incorporate techniques from the SAG algorithm~\citep{DBLP:conf/nips/RouxSB12} into the STORM estimator. Specifically, in each step $t$, we start by randomly sample $i_t$ from the set $\{1,2,\cdots,n\}$. Then, we construct a variance reduction gradient estimator as
\begin{align}\label{FS-v}
\begin{split}
     \v_t = (1-\beta) \v_{t-1}+ \nabla f_{i_t}(\x_t) 
    -(1-\beta)\nabla f_{i_t}(\x_{t-1})  - \beta \left( g_t^{i_t} - \frac{1}{n} \sum_{i=1}^n g_t^i\right),
\end{split}
\end{align}
where the first three terms align with the original STORM method, and the last term, inspired by the SAG algorithm, deals with the finite-sum structure. Here, $g_t$ tracks the gradient as
\begin{equation}\label{FS-g}
\begin{split}
    g_{t+1}^i=\left\{\begin{array}{ll}
\nabla f_{i_t}(\x_t) &  i = i_t \\ 
g_{t}^i  &  i \neq i_t
\end{array}\right. .
\end{split}
\end{equation}
By such a design, we can ensure that the estimation error $\E[\Norm{\v_t - \nabla F(\x_t)}^2]$ reduces gradually. The whole algorithm is stated in Algorithm~\ref{alg:fs}. In this case, we set the hyper-parameters as:
\begin{align*}
     \eta_t = \frac{1}{n^{\frac{1-\alpha}{2}} \left(\sum_{i=1}^t \Norm{\v_i}^{2}\right)^{\alpha}}, \quad \beta = \frac{1}{n},
\end{align*}
where $0 < \alpha <1/3$. The learning rate $\eta_t$ is non-increasing and changes according to the gradient estimations, and the momentum parameter $\beta$ remains unchanged throughout the learning process.
Next, we show that our method enjoys the optimal convergence rate under the following assumptions, which are standard and widely adopted in existing literature~\citep{Fang2018SPIDERNN,Wang2018SpiderBoostAC,pmlr-v139-li21a}.
\begin{algorithm}[t]
	\caption{STORM for Finite-sum Optimization~(SAG-type)}
	\label{alg:fs}
	\begin{algorithmic}[1]
	\STATE {\bfseries Input:} time step $T$, initial point $\x_1$
	\FOR{time step $t = 1$ {\bfseries to} $T$}
        \STATE Sample $i_t$ randomly from $\{1,2,\cdots,n \}$
        \STATE Compute estimator $\v_t$ according to equation~(\ref{FS-v})
        \STATE Update  $g_{t+1}$ according to equation~(\ref{FS-g})
		\STATE Update the decision variable: $\x_{t+1} = \x_t - \eta_t \v_t$
		\ENDFOR
	\STATE Choose $\tau$ uniformly at random from $\{1, \ldots, T\}$
	\STATE Return $\x_\tau$
	\end{algorithmic}
\end{algorithm}
\begin{ass}\label{ass:finite} (Smoothness)
For each $i \in \{1, 2, \cdots, m\}$, function $f_i$ is $L$-smooth such that
\begin{align*}
     \|\nabla f_i(\x) - \nabla f_i(\y)\| \leq L\|\x - \y\|.
\end{align*}
\end{ass}
\begin{ass}\label{asm:stochastic4++} $F_{*}=\inf_{\x} F(\x) \geq-\infty$ and $F\left(\x_{1}\right)-F_{*} \leq \Delta_{F}$ for the initial solution $\x_{1}$.
\end{ass}
\begin{theorem}\label{thm:main_2} Under Assumptions \ref{ass:finite} and \ref{asm:stochastic4++}, our Algorithm~\ref{alg:fs} guarantees that: 
\begin{align*}
   \E\left[ \Norm{\nabla F(\x_{\tau})} \right] \leq \mathcal{O}\left(\frac{n^{1/4}}{T^{1/2}} \left(\Delta_F^{\frac{1}{2(1-\alpha)}}+ L^{\frac{1}{2\alpha}}\right)\right).
\end{align*}
\end{theorem}
\textbf{Remark:} 
Our result matches the lower bound for non-convex finite-sum problems~\cite{Fang2018SPIDERNN,pmlr-v139-li21a}, and makes an improvement over the existing adaptive method, i.e., AdaSpider~\cite{NEURIPS2022_94f625dc}. Specifically, the convergence rate of the AdaSpider algorithm is $\mathcal{O}\left(n^{1 / 4} T^{-1/2} \left(L^2 + \Delta_F \right) \cdot \log \left(1+n T L\right)\right)$, and our result is better than theirs when $\frac{1}{4} < \alpha < \frac{1}{3}$.

We can avoid storing past gradients by following the SVRG method~\citep{NIPS2013_ac1dd209,NIPS:2013:Zhang} to compute the full gradient periodically and incorporate it into STORM estimator. 
Instead of storing the past gradients as in SAG algorithm, we can avoid this storage cost by incorporating elements from the SVRG method. Specifically, we compute a full batch gradient at the first step and every $I$ iteration~(we set $I=n$):
\begin{align*}
    \nabla f(\x_\tau)= \frac{1}{n}\sum_{i=1}^{n} \nabla f_i(\x_{\tau}).
\end{align*}
For other iterations, we randomly select an index \(i_t\) from the set \(\{1, 2, \cdots, n\}\) and compute:
\begin{align}\label{equ:svrg}
    \v_t = (1-\beta) \v_{t-1} +  \nabla f_{i_t}(\x_{t}) -(1-\beta)  \nabla f_{i_t}(\x_{t-1}) -\beta\left( \nabla f_{i_t}(\x_{\tau}) - \nabla f(\x_{\tau})\right).
\end{align}
Note that the first three terms match the original STORM estimator, and the last term, inspired from SVRG, deals with the finite-sum structure. Compared with equation~(\ref{FS-v}) in Algorithm~\ref{alg:fs}, the difference is that we use $\left( \nabla f_{i_t}(\x_{\tau}) - \nabla f(\x_{\tau})\right)$ instead of $\left( g_t^{i_t} - \frac{1}{n} \sum_{i=1}^n g_t^i\right)$ in the last term.
The detailed procedure is outlined in Algorithm~\ref{alg:fs+}. This strategy maintains the same optimal rate, as stated below:
\begin{theorem}\label{T4} Under Assumptions \ref{ass:finite} and \ref{asm:stochastic4++}, our Algorithm~\ref{alg:fs+} guarantees that: 
\begin{align*}
   \E\left[ \Norm{\nabla F(\x_{\tau})} \right] \leq \mathcal{O}\left(\frac{n^{1/4}}{T^{1/2}} \left(\Delta_F^{\frac{1}{2(1-\alpha)}}+ L^{\frac{1}{2\alpha}}\right)\right).
\end{align*}
\end{theorem}
\textbf{Remark:} 
The obtained convergence rate is in the same order as the results in Theorem~\ref{thm:main_2}, and Algorithm~\ref{alg:fs+} does not require storing past gradients anymore.

\begin{algorithm}[!t]
	\caption{STORM for Finite-sum Optimization~(SVRG-type)}
	\label{alg:fs+}
	\begin{algorithmic}[1]
	\STATE {\bfseries Input:} time step $T$, initial point $\x_1$
        \FOR{time step $t = 1$ {\bfseries to} $T$}
        \IF{$t \mod I == 0$}
        \STATE {Set $t= \tau$ and compute $\nabla f(\x_\tau)= \frac{1}{n}\sum_{i=1}^{n} \nabla f_i(\x_{\tau})$}
        \ENDIF
        \STATE Sample $i_t$ randomly from $\{1,2,\cdots,n \}$
        \STATE Compute $\v_t$ according to equation~(\ref{equ:svrg})
		\STATE Update the decision variable: $\x_{t+1} = \x_t - \eta \v_t$
		\ENDFOR
	\STATE Select $\tau$ uniformly at random from $\{1, \ldots, T\}$
	\STATE Return $\x_\tau$
	\end{algorithmic}
\end{algorithm}

\section{Experiments}\label{EX}
In this section, we evaluate the performance of the proposed Ada-STORM method via numerical experiments on image classification tasks and language modeling tasks. In the experiments, we compare our method with STORM~\cite{cutkosky2019momentum}, STORM+~\cite{levy2021storm} and META-STORM~\cite{Liu2022METASTORMGF}, as well as SGD, Adam~\cite{kingma:adam} and AdaBelief~\cite{NEURIPS2020_d9d4f495}. 
We use the default implementation of SGD and Adam from Pytorch~\cite{NEURIPS2019_bdbca288}. For STORM+, we follow its original implementation\footnote{https://github.com/LIONS-EPFL/storm-plus-code}, and build STORM, META-STORM and our Ada-STORM based on it. When it comes to hyper-parameter tuning, we simply set $\alpha=0.3$ for our algorithm. For other methods, we either set the hyper-parameters as recommended in the original papers or tune them by grid search. For example, we search the learning rate of SGD, Adam and AdaBelief from the set $\{1e{-}5,1e{-}4, 1e{-}3, 1e{-}2, 1e{-}1\}$ and select the best one for each method. All the experiments are conducted on eight NVIDIA Tesla V100 GPUs.
\paragraph{Image classification task}
First, we conduct numerical experiments on multi-class image classification tasks to evaluate the performance of the proposed method. Specifically, we train ResNet18 and ResNet34 models~\cite{Resnet18} on the CIFAR-10 and CIFAR-100 datasets~\cite{Krizhevsky2009Cifar10} respectively. For all  optimizers, we set the batch size as 256 and train for 200 epochs. We plot the loss value and the accuracy against the epochs on the CIFAR-10 and CIFAR-100 datasets in Figure~\ref{fig:1} and Figure~\ref{fig:2}. It is observed that, for training loss and training accuracy, our Ada-STORM algorithm achieves comparable performance with respect to other methods, and it outperforms the others in terms of testing loss and thus obtains a better testing accuracy. 
\begin{figure*}[t]
	\centering
	\subfigure[Training loss]{
\includegraphics[width=0.23\textwidth]{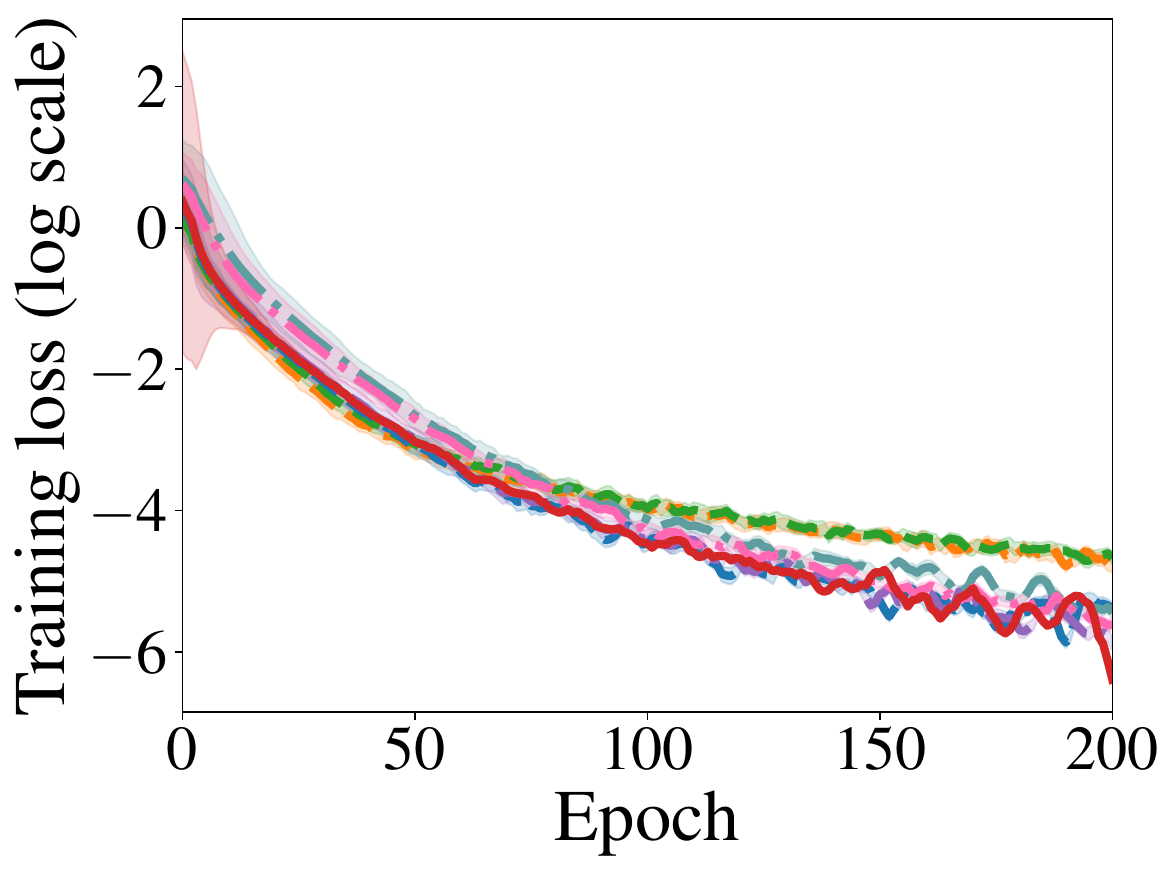}
	}	
	\subfigure[Training accuracy]{
\includegraphics[width=0.23\textwidth]{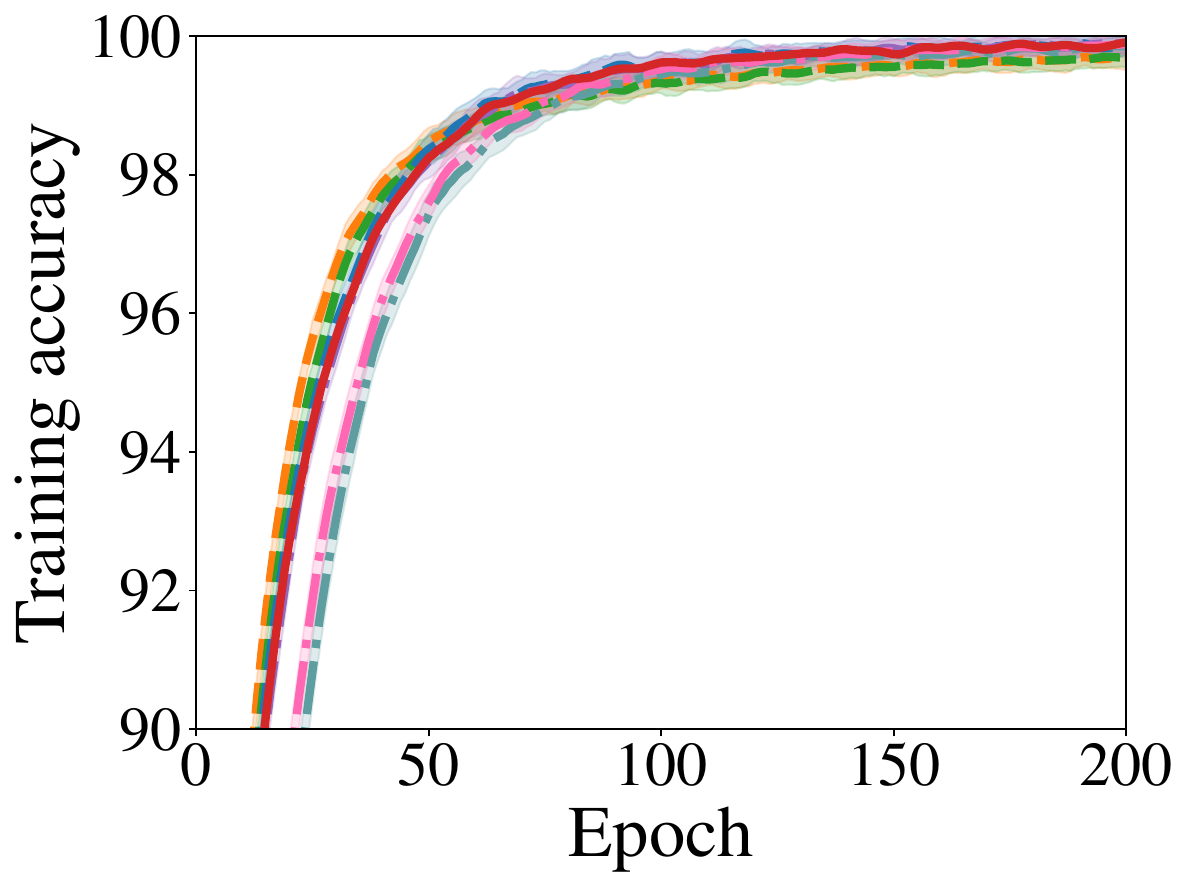}
	}
	\subfigure[Testing loss]{
\includegraphics[width=0.23\textwidth]{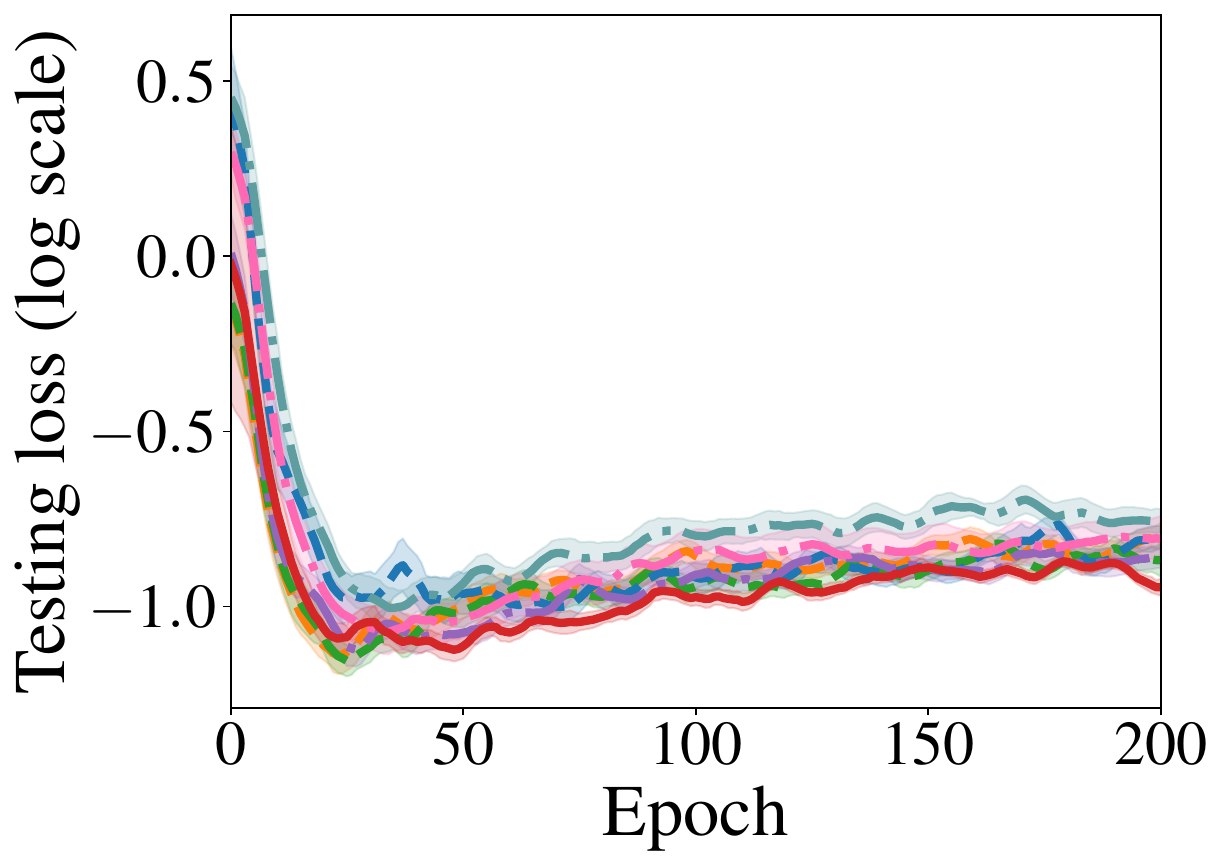}
	}
	\subfigure[Testing accuracy]{
\includegraphics[width=0.23\textwidth]{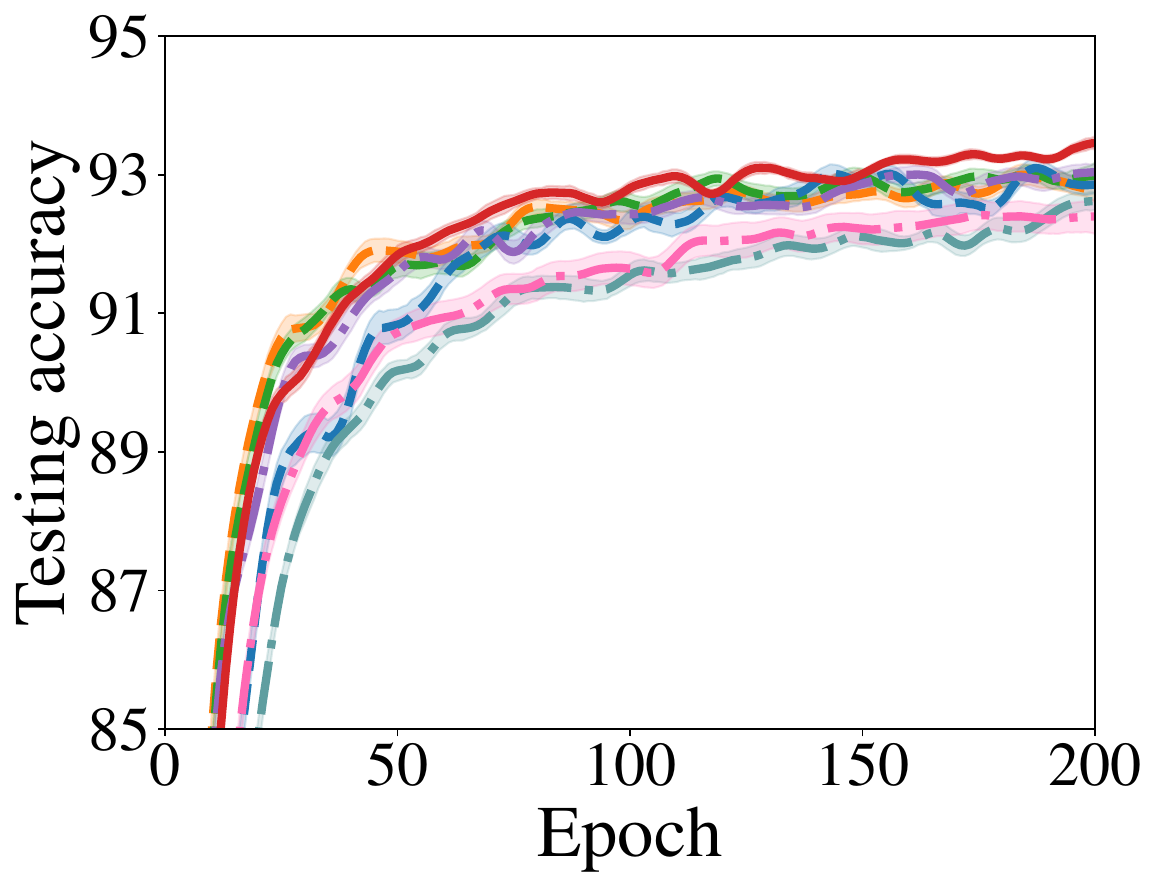}
	}
	\vskip -0.05in
 	\subfigure{
\includegraphics[width=0.99\textwidth]{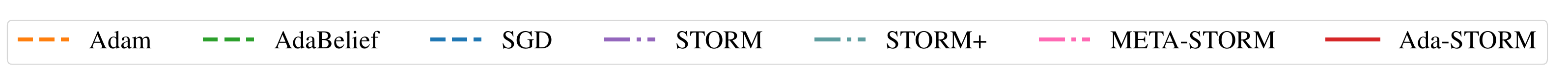}}
\vskip -0.05in
	\caption{Results for CIFAR-10 dataset.}
	\label{fig:1}
\end{figure*}
\begin{figure*}[t]
	\centering
	\subfigure[Training loss]{
\includegraphics[width=0.23\textwidth]{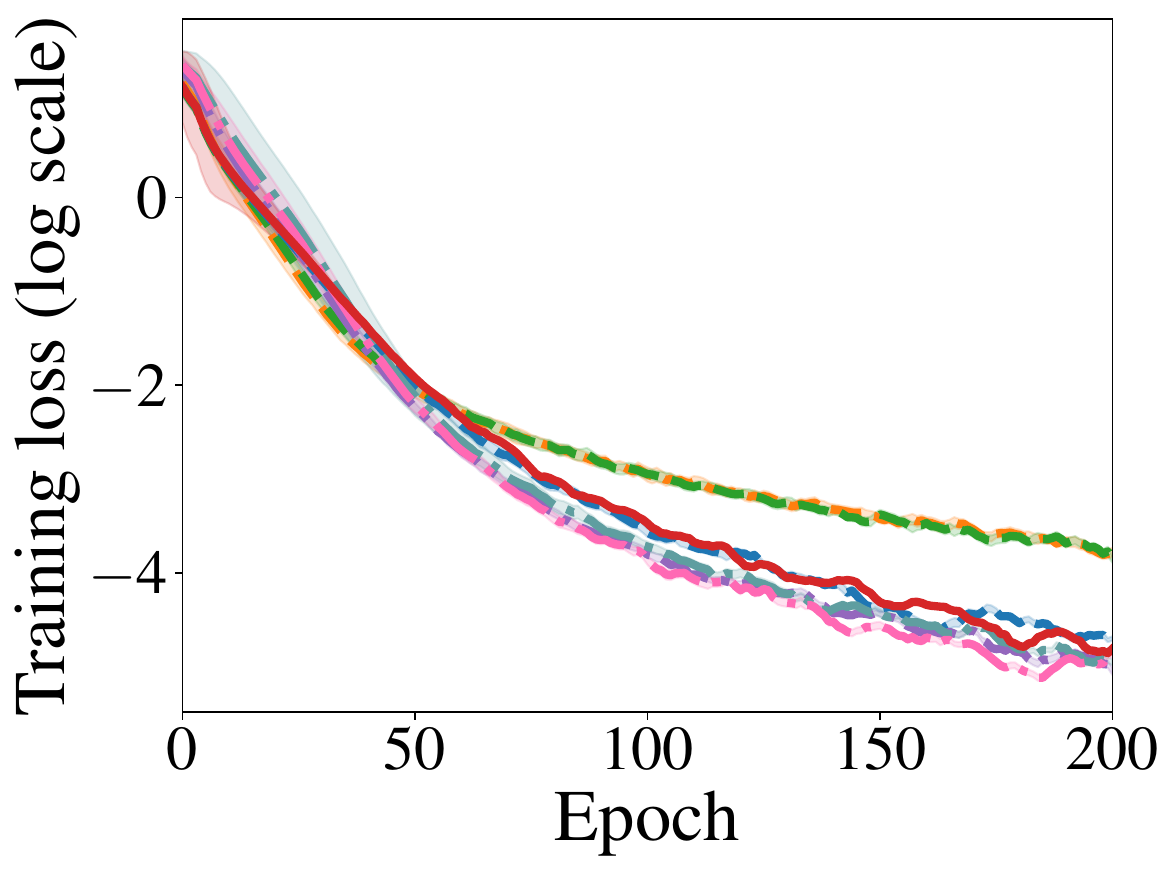}
	}	
	\subfigure[Training accuracy]{
\includegraphics[width=0.23\textwidth]{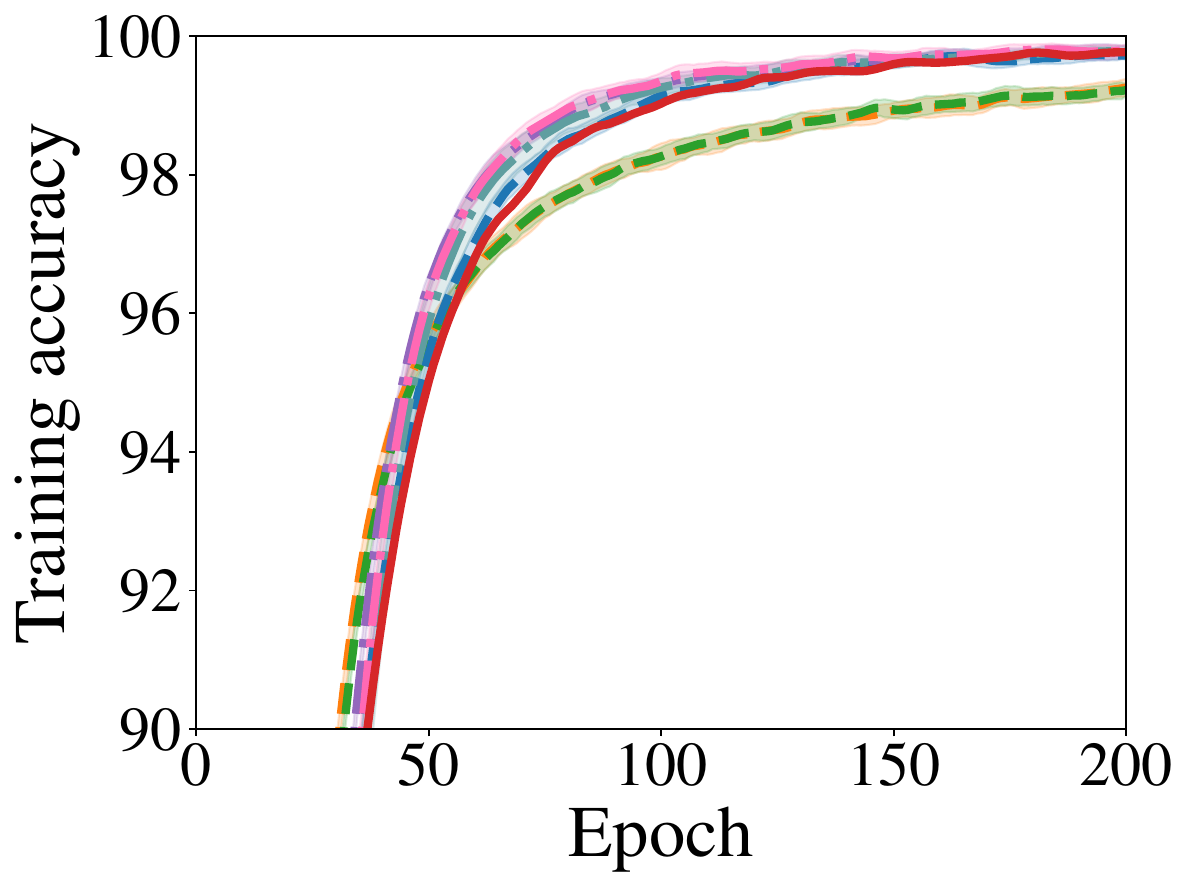}
	}
	\subfigure[Testing loss]{
\includegraphics[width=0.23\textwidth]{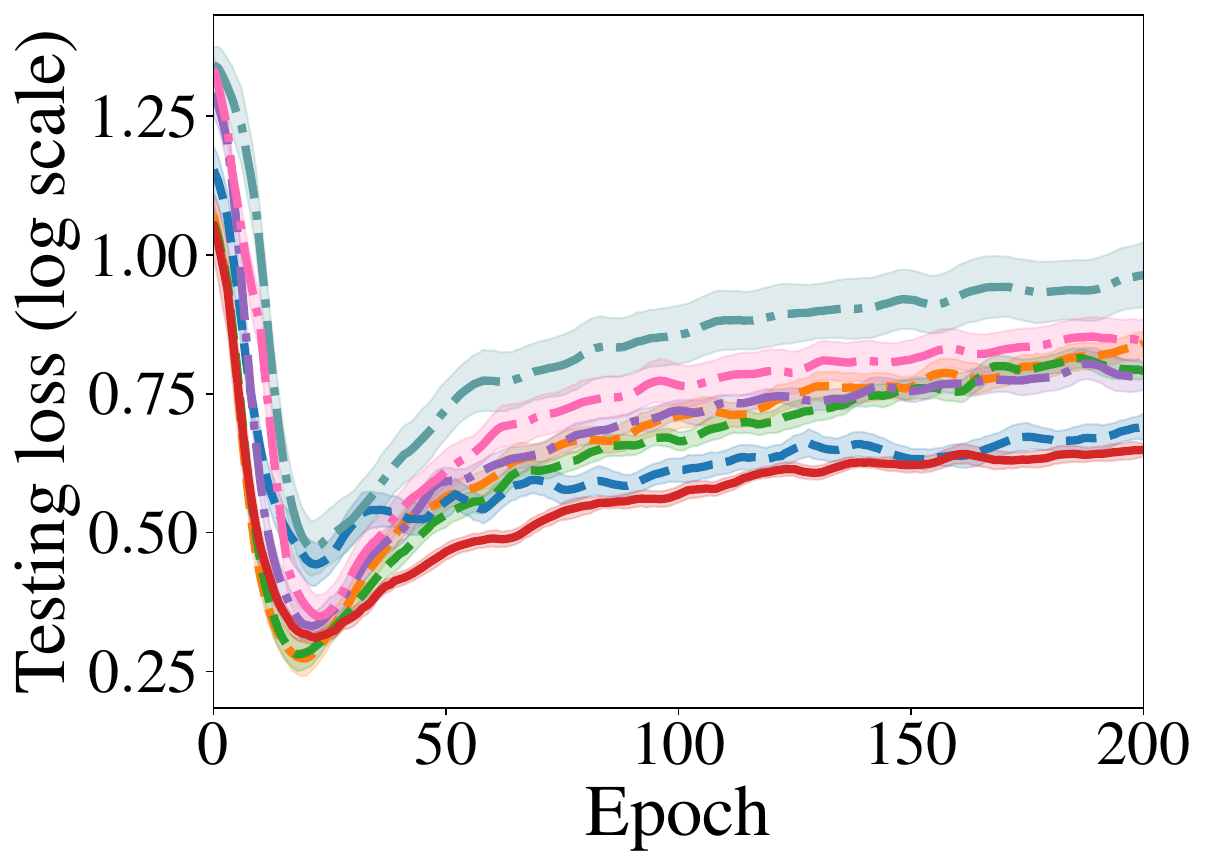}
	}
	\subfigure[Testing accuracy]{
\includegraphics[width=0.23\textwidth]{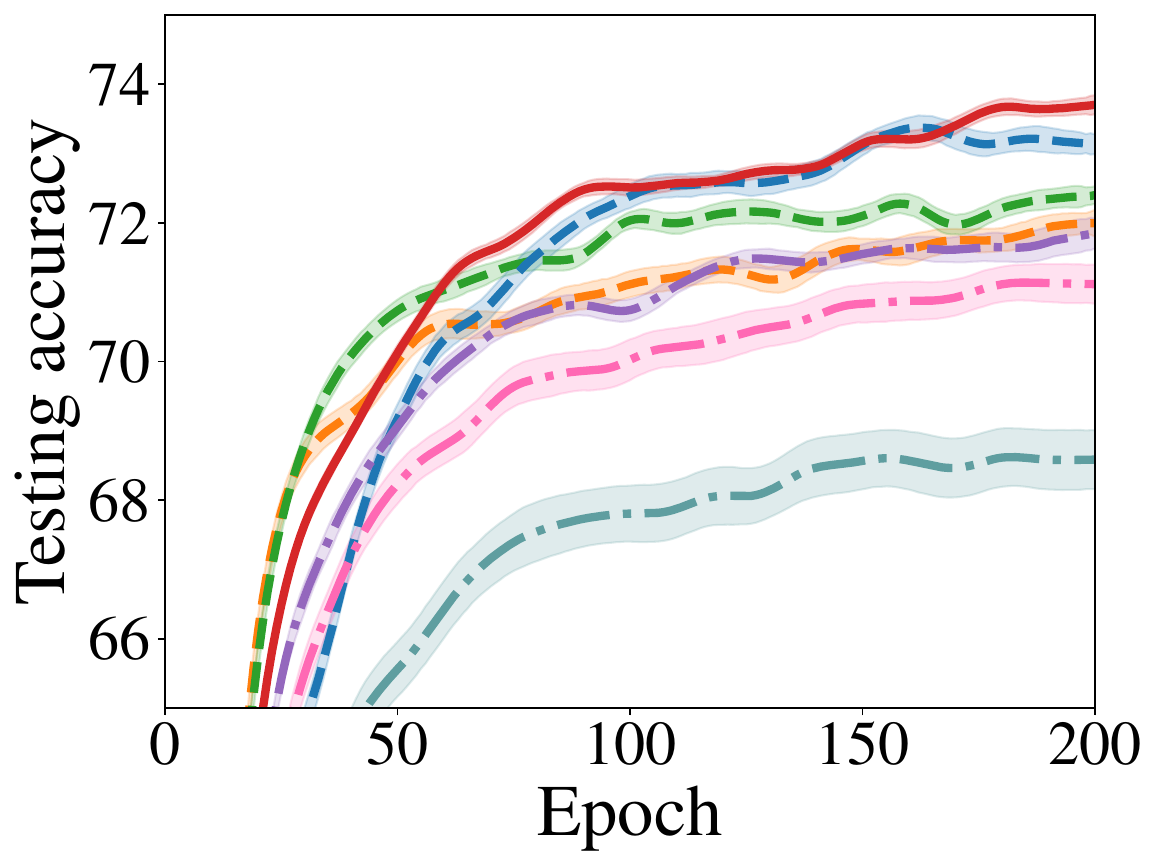}
	}
	\vskip -0.05in
 	\subfigure{
\includegraphics[width=0.99\textwidth]{./Fig/Legend.pdf}}
\vskip -0.05in
	\caption{Results for CIFAR-100 dataset.}
	\label{fig:2}
\end{figure*}

\paragraph{Language modeling task}
Then, we perform experiments on language modeling tasks. Concretely, we train a 2-layer Transformer~\citep{NIPS2017_3f5ee243} over the WiKi-Text2 dataset~\citep{wiki-text}. We use $256$ dimensional word embeddings, $512$ hidden unites and 2 heads. The batch size is set as $20$ and all methods are trained for 40 epochs with dropout rate $0.1$. We also clip the gradients by norm $0.25$ in case of the exploding gradient. We report both the loss and perplexity versus the number of epochs in Figure~\ref{fig:3}. From the results, we observe that our method converges more quickly than other methods and obtains a slightly better perplexity compared with others, indicating the effectiveness of the proposed method.

\begin{figure*}[!t]
	\centering
	\subfigure[Training loss]{
\includegraphics[width=0.23\textwidth]{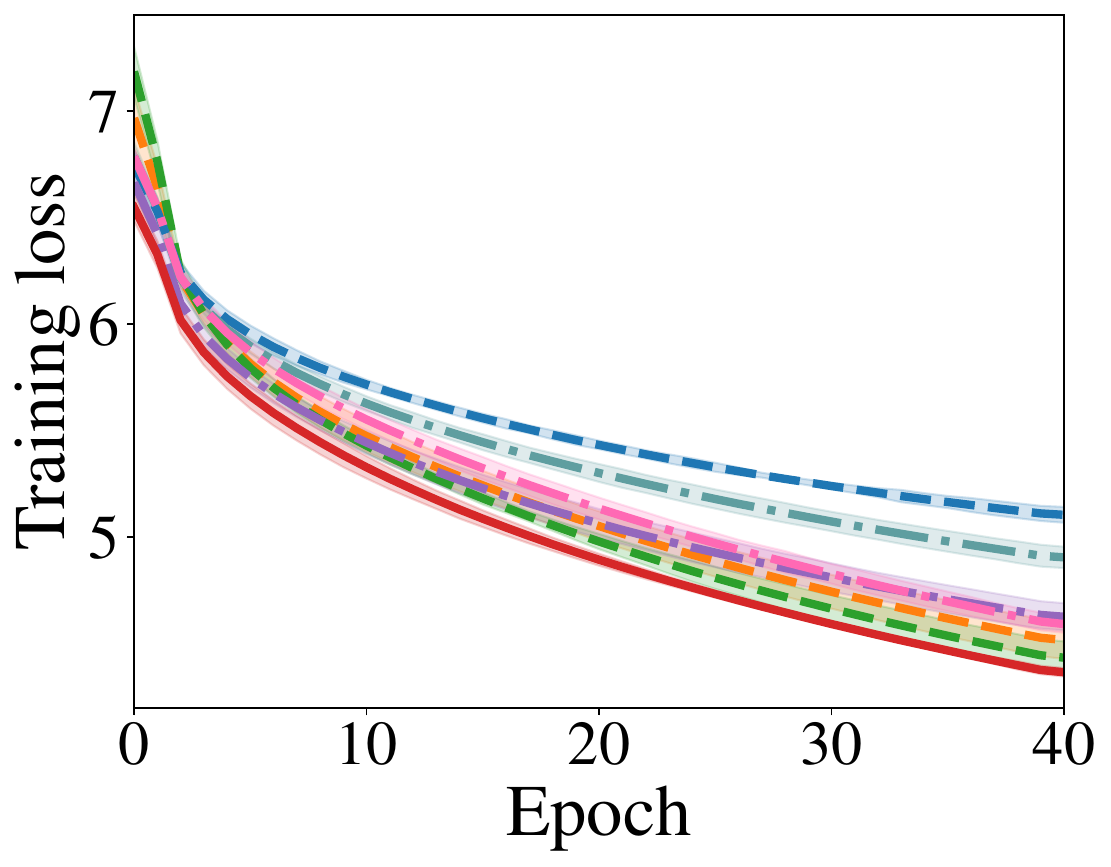}
	}	
	\subfigure[Training perplexity]{
\includegraphics[width=0.23\textwidth]{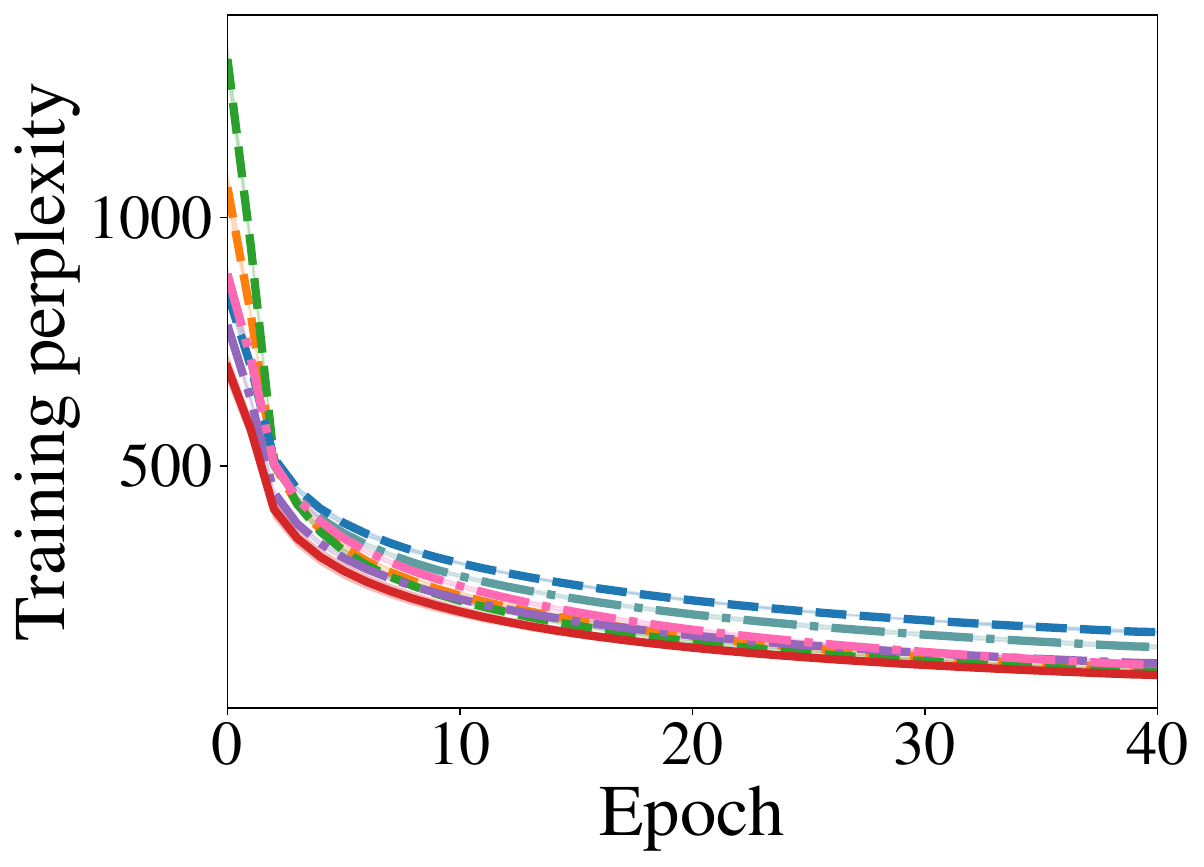}
	}
	\subfigure[Testing loss]{
\includegraphics[width=0.23\textwidth]{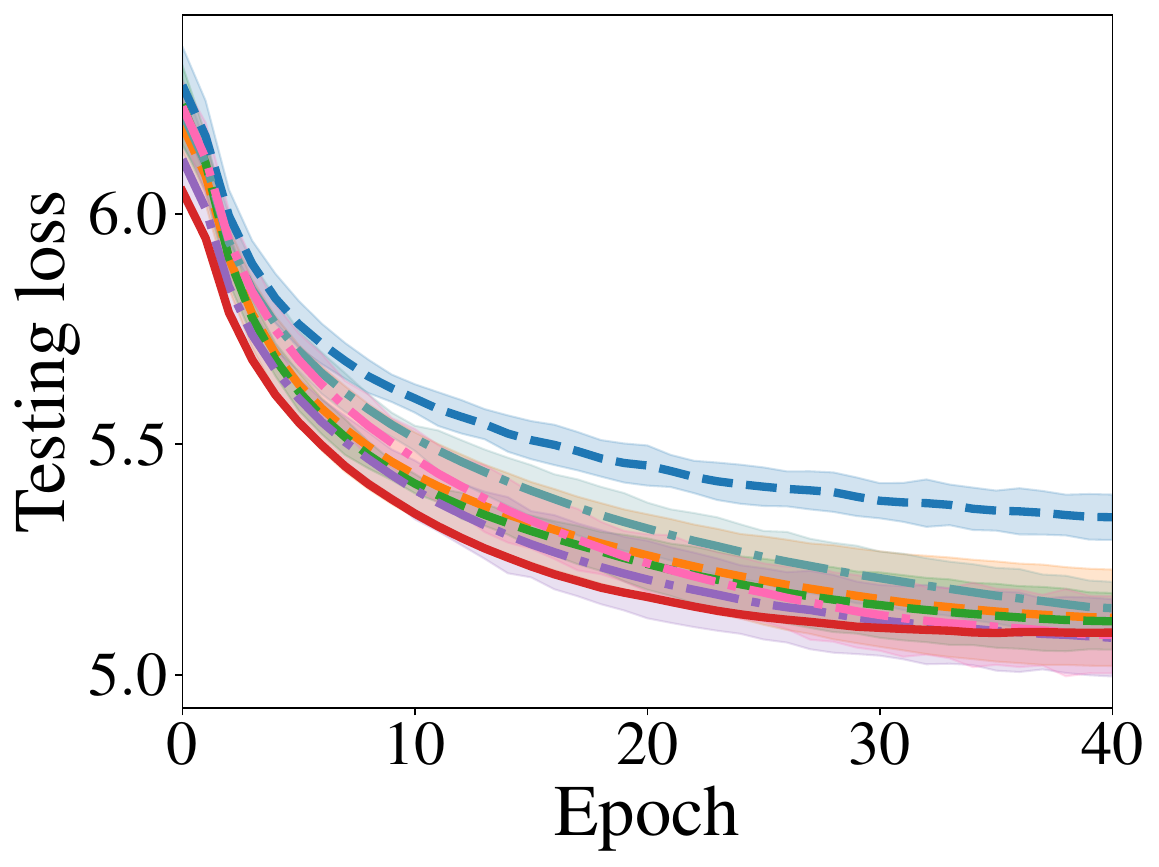}
	}
	\subfigure[Testing perplexity]{
\includegraphics[width=0.23\textwidth]{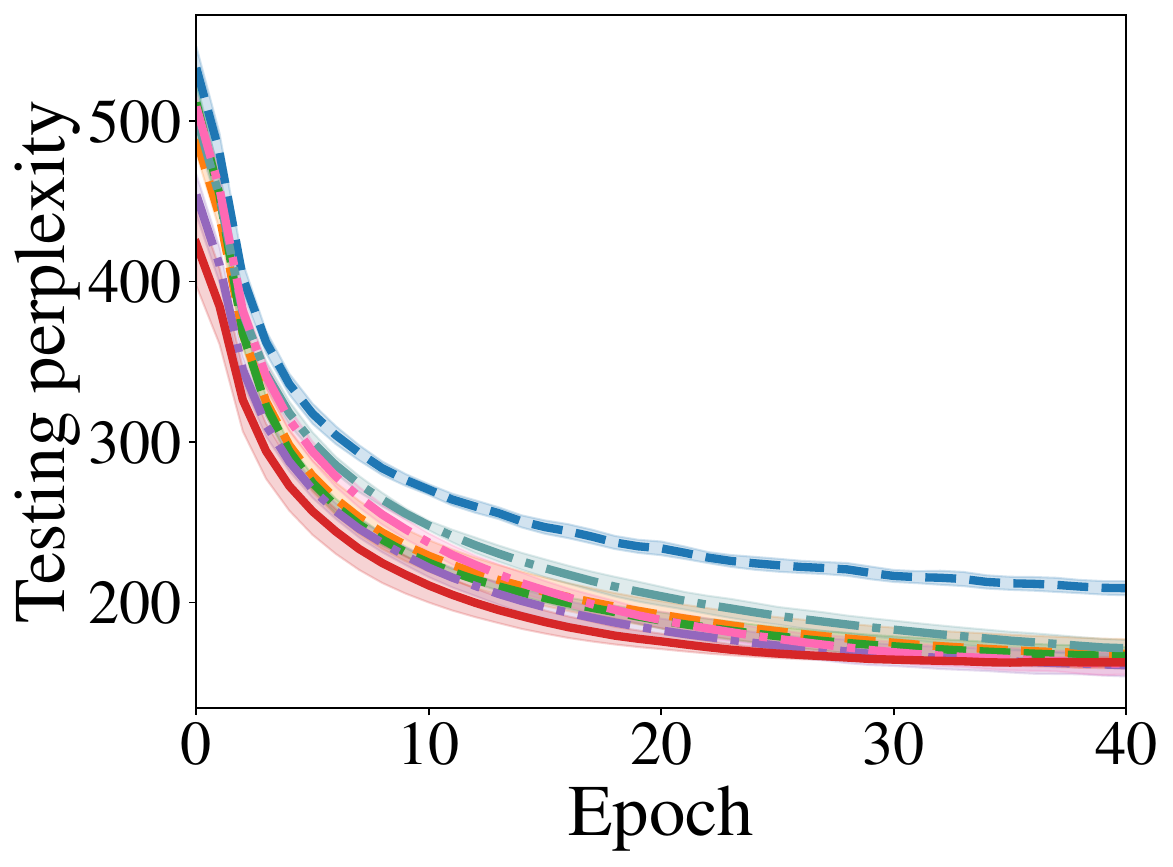}
	}
	\vskip -0.05in
 	\subfigure{
\includegraphics[width=0.99\textwidth]{./Fig/Legend.pdf}}
\vskip -0.05in
	\caption{Results for WikiText-2 dataset.}
	\label{fig:3}
\end{figure*}

\section{Conclusion}
In this paper, we propose an adaptive STORM method to achieve the optimal convergence rate for non-convex functions. Compared with existing methods, our algorithm requires weaker assumptions and does not have the additional $\mathcal {O}(\log T )$ term in the convergence rate. The proposed technique can also be employed to develop optimal adaptive algorithms for compositional optimization. Furthermore, we investigate an adaptive method for non-convex finite-sum optimization, obtaining an improved convergence rate of $\mathcal {O}(n^{1/4} T^{-1/2})$.  Given that STORM algorithm has already been used in many areas such as bi-level optimization~\cite{NEURIPS2021_71cc107d}, federated learning~\cite{das2022faster}, min-max optimization~\cite{NEURIPS2021_d994e372}, sign-based optimization~\cite{jiang2024efficient}, etc., the proposed methods may also inspire the development of adaptive algorithms in these fields.   

\section*{Acknowledgements}
This work was partially supported by National Key R\&D Program of China (2021ZD0112802), NSFC (62122037), and the Postgraduate Research \& Practice Innovation Program of Jiangsu Province (No. KYCX24\_0231).

\bibliography{ref}
\bibliographystyle{abbrvnat}

\newpage
\appendix
\section{Proof of Theorem~\ref{thm:main_0}}
First, we introduce the following lemma, which is frequently used in our proof. 
\begin{lemma}\label{lem1}
Suppose $c_i$ is positive for $i=\left\{1,2,\cdots, n\right\}$, and let $0 < \alpha < 1$. We can ensure that:
\begin{equation*}
\begin{aligned}\left(\sum_{i=1}^{n} c_i\right)^{1-\alpha} \leq \sum_{i=1}^{n}\frac{c_i}{\left(\sum_{j=1}^{i} c_j \right)^\alpha} &\leq \frac{1}{1-\alpha}\left(\sum_{i=1}^{n} c_i\right)^{1-\alpha}.
    \end{aligned}
\end{equation*}
\end{lemma}
\begin{proof}
The proof mainly follows~\citet{McMahan2010AdaptiveBO} and a similar analysis also appears in~\citet{levy2021storm}. 
First, we prove the right part by Induction. 

(1) For $n = 1$, we can easily show that the hypothesis holds:
\begin{align*}
    \frac{c_1}{c_1^\alpha} = c_1^{1-\alpha} \leq \frac{1}{1-\alpha}c_1^{1-\alpha}.
\end{align*}
        
(2) Next, assuming that the hypothesis holds for $n=t-1$, then we show that it also holds for $n=t$. Define $Z = \sum_{i=1}^{t}c_i $ and $X = c_t$. For $n=t$, we have
        \begin{align*}
            \sum_{i=1}^{t}\frac{c_i}{\left(\sum_{j=1}^{i} c_j \right)^\alpha} & =\sum_{i=1}^{t-1}\frac{c_i}{\left(\sum_{j=1}^{i} c_j \right)^\alpha} + \frac{c_t}{\left(\sum_{j=1}^{t} c_j \right)^\alpha}\\
            & \leq \frac{1}{1-\alpha}\left(\sum_{i=1}^{t-1} c_i\right)^{1-\alpha} + \frac{c_t}{\left(\sum_{j=1}^{t} c_j \right)^\alpha}\\
            & = \frac{1}{1-\alpha}(Z-X)^{1-\alpha} + \frac{X}{Z^\alpha} \coloneqq h(X).
        \end{align*}
    Taking the derivative concerning $x$, we know that
        \begin{align*}
        \frac{d h(X)}{d X}=\frac{1}{Z^{\alpha}}-\frac{1}{(Z-X)^{\alpha}},
        \end{align*}
        which indicates that $h(X)$ decreases as $X$ increasing. Since $0 \leq X \leq Z$,
\begin{align*}
    \max _{0 \leq X \leq Z} h(X)=h(0)=\frac{1}{1-\alpha} Z^{1-\alpha}=\frac{1}{1-\alpha}\left(\sum_{i=1}^{t} c_{i}\right)^{1-\alpha},
\end{align*}
        which implies that the hypothesis is true for $n=t$.

Combining (1) and (2), we finish the proof for the right part. Then, we give the proof of the left part as follows:
 \begin{align*}
            \sum_{i=1}^{n}\frac{c_i}{\left(\sum_{j=1}^{i} c_j \right)^\alpha} \geq \sum_{i=1}^{n}\frac{c_i}{\left(\sum_{j=1}^{n} c_j \right)^\alpha} = \left(\sum_{i=1}^{n} c_i\right)^{1-\alpha}.
        \end{align*}
Thus, we finish the proof of this lemma.
\end{proof}
Next, we can  obtain the following guarantee for our algorithm.
\begin{lemma}\label{lem2} Our method enjoys the following guarantee:
     \begin{align*}
     &\sum_{t=1}^T \E \left[{\eta_t}\Norm{\v_t}^2 \right] \\ \leq &\left(2\Delta_f + \sigma^2\right)    +\E\underbrace{\left[\frac{2L^2}{\beta} \sum_{t=1}^T \eta_t^3 \Norm{\v_t}^2\right]}\limits_{\rA}  +\E\underbrace{\left[L \sum_{t=1}^T \eta_t^2\Norm{\v_t}^2\right] }\limits_{\rB} +\E\underbrace{\left[2\beta \sigma^2 \sum_{t=1}^T \eta_t\right]}\limits_{\rC}.
     \end{align*}
 \end{lemma}
\begin{proof}
 According to the definition of estimator $\v_t$, we can deduce that:\

\begin{equation}\label{lemma 6+}
    \begin{aligned}
        &\quad\ \E_{\xi_{t+1}}\left[\Norm{\nabla f(\x_{t+1}) -\v_{t+1}}^2\right]\\
        & = \E_{\xi_{t+1}}\left[\Norm{(1-\beta)\v_{t} +\nabla f(\x_{t+1};\xi_{t+1}) -  (1-\beta)\nabla f(\x_{t};\xi_{t+1}) - \nabla f(\x_{t+1})}^2\right]\\
        & = \E_{\xi_{t+1}}\left[\|(1-\beta)(\v_{t} - \nabla f(\x_{t})) + \left(\nabla f(\x_{t})-\nabla f(\x_{t+1}) + \nabla f(\x_{t+1};\xi_{t+1}) - \nabla f(\x_{t};\xi_{t+1}) \right) \right. \\
        & \qquad \left.  + \beta\left(\nabla f(\x_{t};\xi_{t+1}) - \nabla f(\x_{t})  \right) \|^2 \right]\\
          & \leq  \mathbb{E}_{\xi_{t+1}}\left[(1-\beta)^2||\mathbf{v}_{t} - \nabla f(\mathbf{x}_{t})||^2\right]  +\mathbb{E}_{\xi_{t+1}}\left[|| \nabla f(\mathbf{x}_{t})-\nabla f(\mathbf{x}_{t+1}) \right. \\
           & \qquad  \left. + \nabla f(\mathbf{x}_{t+1};\xi_{t+1}) - \nabla f(\mathbf{x}_{t};\xi_{t+1}) + \beta\left(\nabla f(\mathbf{x}_{t};\xi_{t+1}) - \nabla f(\mathbf{x}_{t})  \right) ||^2 \right] \\
        & \leq  \mathbb{E}_{\xi_{t+1}}\left[(1-\beta)^2||\mathbf{v}_{t} - \nabla f(\mathbf{x}_{t})||^2\right] \\
        &\quad +2 \mathbb{E}_{\xi_{t+1}}\left[|| \nabla f(\mathbf{x}_{t})-\nabla f(\mathbf{x}_{t+1}) + \nabla f(\mathbf{x}_{t+1};\xi_{t+1}) - \nabla f(\mathbf{x}_{t};\xi_{t+1}) ||^2\right] \\
        & \quad  + 2\mathbb{E}_{\xi_{t+1}}\left[|| \beta\left(\nabla f(\mathbf{x}_{t};\xi_{t+1}) - \nabla f(\mathbf{x}_{t})  \right)||^2\right] \\
        & \leq (1-\beta)^2 \E_{\xi_{t+1}}\left[\|\v_{t} - \nabla f(\x_{t})\|^2\right] + 2\beta^2 \E_{\xi_{t+1}}\left[\|\nabla f(\x_{t};\xi_{t+1}) - \nabla f(\x_{t})\|^2\right] \\
        & \qquad +2 \E_{\xi_{t+1}}\left[\|\nabla f(\x_{t+1};\xi_{t+1}) - \nabla f(\x_{t};\xi_{t+1}) \|^2\right]\\
        & \leq  (1-\beta)\|\v_{t} - \nabla f(\x_{t})\|^2  + 2\beta^2\sigma^2 + 2L^2 \|\x_{t+1} - \x_{t} \|^2\\
        & = (1-\beta)\|\v_{t} - \nabla f(\x_{t})\|^2\  + 2\beta^2\sigma^2 + 2L^2 \eta_{t}^2\|\v_{t} \|^2 .
    \end{aligned}
\end{equation}

Note that $\eta_t$ is independent of random variable $\xi_{t+1}$. So we can guarantee that:
    \begin{align*}
    \E_{\xi_{t+1}}\left[\eta_t\Norm{\nabla f(\x_{t+1}) -\v_{t+1}}^2\right] \leq (1-\beta)\eta_t\|\v_{t} - \nabla f(\x_{t})\|^2\  + 2\beta^2\sigma^2 \eta_t+ 2L^2 \eta_{t}^3\|\v_{t} \|^2.
    \end{align*}
After rearranging, we have:
    \begin{align*}
    &\eta_t\Norm{\v_t - \nabla f(\x_{t})}^2
       \\ \leq& \frac{\eta_t}{\beta}\Norm{\v_t - \nabla f(\x_{t})}^2   - \E_{\xi_{t+1}}\left[\frac{\eta_t}{\beta}\Norm{\v_{t+1} - \nabla f(\x_{t+1})}^2 \right]+  2\beta   \sigma^2 \eta_t+ \frac{2L^2}{\beta}\eta_t^3 \Norm{\v_t}^2.
    \end{align*}
Letting $\H_t$ be the history to time $t$, i.e., $\H_t = \{\xi_1,\cdots,\xi_t \}$, we ensure that
    \begin{align*}
    \E_{\H_t}\left[\eta_t\Norm{\v_t - \nabla f(\x_{t})}^2\right]
        \leq &\E_{\H_t}\left[\frac{\eta_t}{\beta}\Norm{\v_t - \nabla f(\x_{t})}^2\right]   - \E_{\H_{t+1}}\left[\frac{\eta_t}{\beta}\Norm{\v_{t+1} - \nabla f(\x_{t+1})}^2 \right]\\
        &+  2\beta   \sigma^2 \E_{\H_t}\left[\eta_t\right]+ \frac{2L^2}{\beta}\E_{\H_t}\left[\eta_t^3 \Norm{\v_t}^2\right].
    \end{align*}
    By summing up and noting that $\eta_t$ is non-increasing such that $\eta_{t+1} \leq \eta_t$, we have:
    \begin{align}\label{LL4}
        \begin{split}
            &\sum_{t=1}^{T}\E_{\H_t}\left[\eta_t\Norm{\v_t - \nabla f(\x_{t})}^2\right] \\
        \leq  &\E_{\H_1}\left[\frac{\eta_1}{\beta}\Norm{\v_1 - \nabla f(\x_{1})}^2\right] 
        +  2\beta \sigma^2 \sum_{t=1}^T \E_{\H_t}\left[\eta_t \right]  + \frac{2L^2}{\beta}\sum_{t=1}^T\E_{\H_t}\left[ \eta_t^3\Norm{\v_t}^2 \right].
        \end{split}
    \end{align}
Since we use a large batch size in the first iteration, that is, $B_0 = T^{1/3}$, we can now ensure that   $\E_{\H_1}\left[\Norm{\v_1 - \nabla f(\x_{1})}^2\right] \leq \frac{\sigma^2}{B_0} = \frac{\sigma^2}{T^{1/3}}$. Due to the fact that $\eta_1 \leq T^{-1/3}$ and $\beta = T^{-2/3}$, the first term of the above inequality is less than $\sigma^2$. So, we can finally have:
    \begin{align}\label{144}
        \sum_{t=1}^{T}\E\left[\eta_t\Norm{\v_t - \nabla f(\x_{t})}^2\right]
        \leq  \sigma^2
        +  2\beta \sigma^2 \sum_{t=1}^T \E\left[\eta_t \right]  + \frac{2L^2}{\beta}\sum_{t=1}^T\E\left[ \eta_t^3\Norm{\v_t}^2 \right].
    \end{align}
Also, due to the smoothness of $f (\x)$, we  know that:
\begin{align*}
    f(\x_{t+1}) & \leq f(\x_{t}) +  \left\langle\nabla f\left(\mathbf{x}_{t}\right), \mathbf{x}_{t+1}-\mathbf{x}_{t}\right\rangle+\frac{L}{2}\left\|\mathbf{x}_{t+1}-\mathbf{x}_{t}\right\|^{2}\\
    & = f\left(\mathbf{x}_{t}\right)-\eta_{t}\left\langle\nabla f\left(\mathbf{x}_{t}\right), \mathbf{v}_{t}\right\rangle+\frac{\eta_{t}^{2} L}{2}\left\|\mathbf{v}_{t}\right\|^{2}\\
    & = f\left(\mathbf{x}_{t}\right)-\eta_{t}\left\langle\nabla f\left(\mathbf{x}_{t}\right), \mathbf{v}_{t}\right\rangle+\frac{\eta_{t}}{2}\left\|\nabla f\left(\mathbf{x}_{t}\right)\right\|^{2}+\frac{\eta_{t}}{2}\left\|\mathbf{v}_{t}\right\|^{2}-\frac{\eta_{t}}{2}\left\|\nabla f\left(\mathbf{x}_{t}\right)\right\|^{2}\\
    &\quad\quad -\frac{\eta_{t}}{2}\left\|\mathbf{v}_{t}\right\|^{2}+\frac{\eta_{t}^{2} L}{2}\left\|\mathbf{v}_{t}\right\|^{2}\\
    &= f(\x_t) +\frac{\eta_t}{2} \Norm{\nabla f(\x_t) -\v_t}^2 - \frac{\eta_t}{2} \Norm{\nabla f(\x_t)}^2 -\frac{\eta_t}{2}\Norm{\v_t}^2 +\frac{\eta_t^2 L}{2}\Norm{\v_t}^2.
\end{align*}
By summing up and re-arranging, we have:
\begin{align}\label{smooth}
\begin{split}
        \sum_{t=1}^T {\eta_t}\Norm{\v_t}^2  \leq 2f(\x_1) -2f(\x_{T+1})+ \sum_{t=1}^T {\eta_t}\Norm{\nabla f(\x_t) -\v_t}^2   +\sum_{t=1}^T{\eta_t^2 L} \Norm{\v_t}^2.
\end{split}
\end{align}
Then, by using equation~(\ref{144}) and the fact that $f(\x_1) - f_{*} \leq \Delta_f$, we have:
\begin{align*}
    \sum_{t=1}^T \E \left[{\eta_t}\Norm{\v_t}^2 \right] \leq 2\Delta_f + \sigma^2 +2\beta \sigma^2  \sum_{t=1}^T \E\left[\eta_t\right] + 2L^2 \sum_{t=1}^T   \E \left[\frac{\eta_t^3}{\beta}\Norm{\v_t}^2 \right]  +L\sum_{t=1}^T \E \left[\eta_t^2  \Norm{\v_t}^2 \right],
    \end{align*}
which finishes the proof of Lemma~\ref{lem2}.
\end{proof}
To effectively bound each term in the above lemma, we divide the algorithm into two stages. Suppose that starting from iteration $t=s$, the condition $\sum_{i=1}^{t} \Norm{\v_i}^2 \geq T^{1/3}$ begins to hold. We refer to iterations $t=\{1,2,\cdots,s-1\}$ as the first stage, and $t= \{s,\cdots, T\}$ as the second stage.

\textbf{Bounding LHS:} In the first stage,  we know that
\begin{align*}
    \sum_{t=1}^{s-1} \eta_t\Norm{\v_t}^2 = \frac{1}{T^{1/3}} \sum_{t=1}^{s-1} \Norm{\v_t}^2.
\end{align*}
For the second stage, our analysis leads to:
\begin{align*}
    \sum_{t=s}^T \eta_t\Norm{\v_t}^2   &= \sum_{t=s}^T \frac{\Norm{\v_t}^2}{T^{\frac{1-\alpha}{3}}\left(\sum_{i=1}^t \Norm{\v_i}^2\right)^\alpha}\\
    \geq & \sum_{t=s}^T \frac{\Norm{\v_t}^2}{T^{\frac{1-\alpha}{3}}\left(T^{\frac{1}{3}}+\sum_{i=s}^t \Norm{\v_i}^2\right)^\alpha} \\
    \geq & \sum_{t=s}^T \frac{\Norm{\v_t}^2}{T^{\frac{1-\alpha}{3}}  \left(T^{\frac{\alpha}{3}}+\left(\sum_{i=s}^t \Norm{\v_i}^2\right)^\alpha \right)} \\
    \geq & \sum_{t=s}^T \frac{\Norm{\v_t}^2}{T^{\frac{1-\alpha}{3}} \cdot 2\max\left\{T^{\frac{\alpha}{3}},\left(\sum_{i=s}^t \Norm{\v_i}^2\right)^\alpha \right\}} \\
     \geq & \frac{1}{2 T^{\frac{1-\alpha}{3}}}\min \left\{ \frac{1}{T^\frac{\alpha}{3}}\sum_{t=s}^T \Norm{\v_t}^2, \left(\sum_{t=s}^T \Norm{\v_t}^2 \right)^{1-\alpha}\right\} \\
     = &\frac{1}{2}\underbrace{\min \left\{ \frac{1}{T^{1/3}} \sum_{t=s}^T \Norm{\v_t}^2,\left(\frac{1}{T^{1/3}}\sum_{t=s}^T \Norm{\v_t}^2 \right)^{1-\alpha}\right\}}\limits_{\coloneqq \Gamma},
    \end{align*}
where the first inequality stems from $\sum_{t=1}^{s-1} \Norm{\v_t}^2\leq T^{1/3}$, the second inequality results from $(x+y)^\alpha \leq x^\alpha + y^\alpha$ for positive $x,y$ and $0 < \alpha < 1/3$, and the forth inequality applies Lemma~\ref{lem1}. Next, we bound  $\rA,\rB,\rC$ as follows.

\textbf{Bounding $\rA$:}
In the first stage, with $\eta_t = T^{-1/3}$, $\beta = T^{-2/3}$, and  $\sum_{t=1}^{s-1} \Norm{\v_t}^2 \leq T^{1/3}$, we can derive:
\begin{align*}
    \frac{2L^2}{\beta}\sum_{t=1}^{s-1} \eta_t^3\Norm{\v_t}^2 = \frac{2L^2}{T^{1/3}}   \sum_{t=1}^{s-1} \Norm{\v_t}^2 \leq 2L^2.
\end{align*}
For the second stage, the analysis gives:
\begin{align*}
    \frac{2L^2}{\beta}\sum_{t=s}^T \eta_t^3\Norm{\v_t}^2 \leq 2L^2  \sum_{t=s}^T \frac{\Norm{\v_t}^2}{T^{\frac{1-3\alpha}{3}}\left(\sum_{i=s}^t\Norm{\v_i}^2\right)^{3\alpha}}
     \leq \frac{2L^2}{1-3\alpha} \left(\frac{1}{T^{1/3}}\sum_{t=s}^T \Norm{\v_t}^2\right)^{1-3\alpha},
\end{align*}
where the second inequality uses Lemma~\ref{lem1}. Then we also have:
\begin{equation*}
    \begin{split}
       \frac{2L^2}{1-3\alpha} \left(\frac{1}{T^{1/3}}\sum_{t=s}^T \Norm{\v_t}^2\right)^{1-3\alpha} 
        =&  \frac{2L^2}{1-3\alpha}(8-24\alpha)^{1-3\alpha}\left(\frac{1}{(8-24\alpha)}\frac{1}{T^{1/3}}\sum_{t=s}^T \Norm{\v_t}^2\right)^{1-3\alpha}\\
     \leq & 3\alpha\left(\frac{2L^2}{1-3\alpha}(8-24\alpha)^{1-3\alpha}\right)^{\frac{1}{3\alpha}}+\frac{1}{8 T^{1/3}}\sum_{t=s}^T \Norm{\v_t}^2, 
    \end{split}
\end{equation*}
where the  inequality  employs Young's inequality, such that $xy \leq 3\alpha x^{\frac{1}{3\alpha}} + (1-3\alpha)y^{\frac{1}{1-3\alpha}}$ for positive $x, y$. Very similarly, we have:
\begin{equation*}
    \begin{split}
     &\frac{2L^2}{1-3\alpha} \left(\frac{1}{T^{1/3}}\sum_{t=s}^T \Norm{\v_t}^2\right)^{1-3\alpha} \\
     =& \frac{2L^2}{1-3\alpha}\left(\frac{8-24\alpha}{1-\alpha}\right)^{\frac{1-3\alpha}{1-\alpha}}
     \left(\frac{1-\alpha}{8-24\alpha}\right)^{\frac{1-3\alpha}{1-\alpha}} \left(\frac{1}{T^{1/3}}\sum_{t=s}^T \Norm{\v_t}^2\right)^{1-3\alpha}\\
     \leq & \frac{2\alpha}{1-\alpha}\left(\frac{2L^2}{1-3\alpha}\left(\frac{8-24\alpha}{1-\alpha}\right)^{\frac{1-3\alpha}{1-\alpha}} \right)^{\frac{1-\alpha}{2\alpha}}+\frac{1}{8}\left(\frac{1}{ T^{1/3}}\sum_{t=s}^T \Norm{\v_t}^2\right)^{1-\alpha},
    \end{split}
\end{equation*}
where the second inequality employs Young's inequality, such that $xy \leq \frac{2\alpha}{1-\alpha} x^{\frac{1-\alpha}{2\alpha}} + \frac{1-3\alpha}{1-\alpha} y^{\frac{1-\alpha}{1-3\alpha}}$ for positive $x, y$. Combining all above, we know that
\begin{align*}
    \rA \leq &2L^2 + 3\alpha\left(\frac{2L^2}{1-3\alpha}(8-24\alpha)^{1-3\alpha}\right)^{\frac{1}{3\alpha}}\\
    &+\frac{2\alpha}{1-\alpha}\left(\frac{2L^2}{1-3\alpha}\left(\frac{8-24\alpha}{1-\alpha}\right)^{\frac{1-3\alpha}{1-\alpha}} \right)^{\frac{1-\alpha}{2\alpha}}+ \frac{\Gamma}{8} .
\end{align*}

\textbf{Bounding $\rB$:}
In the first stage, with the learning rate set at $\eta_t = T^{-1/3}$ and $\sum_{t=1}^{s-1} \Norm{\v_t}^2 \leq T^{1/3}$, we observe:
\begin{align*}
    L\sum_{t=1}^{s-1} \eta_t^2\Norm{\v_t}^2 = \frac{L}{T^{2/3}}   \sum_{t=1}^{s-1} \Norm{\v_t}^2 \leq \frac{L}{T^{1/3}} \leq L.
\end{align*}

For the second stage, our analysis reveals:
\begin{align*}
    L\sum_{t=s}^T \eta_t^2\Norm{\v_t}^2 \leq  &L  \sum_{t=s}^T \frac{\Norm{\v_t}^2}{T^{\frac{2(1-\alpha)}{3}}\left(\sum_{i=s}^t\Norm{\v_i}^2\right)^{2\alpha}} \leq  \frac{L}{1-2\alpha} \left(\frac{1}{T^{1/3}}\sum_{t=s}^T \Norm{\v_t}^2\right)^{1-2\alpha},
\end{align*}
where the second inequality leverages Lemma~\ref{lem1}, and we have:
\begin{equation*}
    \begin{split}
        &\frac{L}{1-2\alpha} \left(\frac{1}{T^{1/3}}\sum_{t=s}^T \Norm{\v_t}^2\right)^{1-2\alpha} \\
        =&\frac{L}{1-2\alpha}(8-16\alpha)^{1-2\alpha}\frac{1}{(8-16\alpha)^{1-2\alpha}} \left(\frac{1}{T^{1/3}}\sum_{t=s}^T \Norm{\v_t}^2\right)^{1-2\alpha}\\
    \leq & 2\alpha \left(\frac{(8-16\alpha)^{1-2\alpha}L}{1-2\alpha} \right)^{\frac{1}{2\alpha}} + \frac{1}{8T^{1/3}}\sum_{t=s}^T \Norm{\v_t}^2,
    \end{split}
\end{equation*}
where the inequality is due to Young's inequality, such that $xy \leq 2\alpha x^{\frac{1}{2\alpha}} +(1-2\alpha) y^{\frac{1}{1-2\alpha}}$ for positive $x, y$. Very similarly, we have:
\begin{align*}
    &\frac{L}{1-2\alpha} \left(\frac{1}{T^{1/3}}\sum_{t=s}^T \Norm{\v_t}^2\right)^{1-2\alpha} \\
    =& \frac{L}{1-2\alpha}\left(\frac{8-16\alpha}{1-\alpha}\right)^{\frac{1-2\alpha}{1-\alpha}}\left(\frac{1-\alpha}{8-16\alpha}\right)^{\frac{1-2\alpha}{1-\alpha}} \left(\frac{1}{T^{1/3}}\sum_{t=s}^T \Norm{\v_t}^2\right)^{1-2\alpha} \\
    \leq & \frac{\alpha}{1-\alpha}\left(\frac{L}{1-2\alpha} \left(\frac{8-16\alpha}{1-\alpha}\right)^{\frac{1-2\alpha}{1-\alpha}}\right)^{\frac{1-\alpha}{\alpha}} + \frac{1}{8}\left(\frac{1}{T^{1/3}}\sum_{t=s}^T \Norm{\v_t}^2\right)^{1-\alpha},
\end{align*}
where the second inequality employs Young's inequality, such that $xy \leq \frac{\alpha}{1-\alpha} x^{\frac{1-\alpha}{\alpha}} + \frac{1-2\alpha}{1-\alpha} y^{\frac{1-\alpha}{1-2\alpha}}$ for positive $x, y$. Combining all the above, we know that
\begin{align*}
    \rB & \leq L +2\alpha \left(\frac{(8-16\alpha)^{1-2\alpha}L}{1-2\alpha} \right)^{\frac{1}{2\alpha}} +\frac{\alpha}{1-\alpha}\left(\frac{L}{1-2\alpha} \left(\frac{8-16\alpha}{1-\alpha}\right)^{\frac{1-2\alpha}{1-\alpha}}\right)^{\frac{1-\alpha}{\alpha}}+ \frac{\Gamma}{8}.
\end{align*}

\textbf{Bounding $\rC$:} Given that $\beta = T^{-2/3}$ and $\eta_t \leq T^{-1/3}$, we can easily know that $\rC \leq 2 \sigma^2$.

So far, we bound all terms in Lemma~\ref{lem2}, and we can deduce 
\begin{align*}
     \textbf{LHS} &\geq \E\left[\frac{1}{T^{1/3}} \sum_{t=1}^{s-1} \Norm{\v_t}^2\right] + \E\left[\frac{\Gamma}{2}\right] ;\quad
     \textbf{RHS} \leq \E\left[\frac{\Gamma}{4}\right] + C_0,
\end{align*}
where 
\begin{equation}\label{C0}
    \begin{split}
        C_0 = &\left(2\Delta_f + 3\sigma^2 + L+2L^2 \right) +3\alpha\left(\frac{2L^2}{1-3\alpha}\left(8-24\alpha\right)^{1-3\alpha}\right)^{\frac{1}{3\alpha}}\\
    &+\frac{2\alpha}{1-\alpha}\left(\frac{2L^2}{1-3\alpha}\left(\frac{8-24\alpha}{1-\alpha}\right)^{\frac{1-3\alpha}{1-\alpha}} \right)^{\frac{1-\alpha}{2\alpha}}+2\alpha \left(\frac{(8-16\alpha)^{1-2\alpha}L}{1-2\alpha} \right)^{\frac{1}{2\alpha}} \\
    & +\frac{\alpha}{1-\alpha}\left(\frac{L}{1-2\alpha} \left(\frac{8-16\alpha}{1-\alpha}\right)^{\frac{1-2\alpha}{1-\alpha}}\right)^{\frac{1-\alpha}{\alpha}}.
    \end{split}
\end{equation}
These suggest that 
\begin{align*}
     \E\left[\sum_{t=1}^{s-1} \Norm{\v_t}^2\right] \leq C_0 T^{1/3} ;\quad
     \E\left[\Gamma\right] \leq  4C_0.
\end{align*}
With the definition of $\Gamma$, we know that
\begin{align*}
    &\E\left[ \min \left\{ \frac{1}{T^{1/3}} \sum_{t=s}^T \Norm{\v_t}^2,\left(\frac{1}{T^{1/3}}\sum_{t=s}^T \Norm{\v_t}^2 \right)^{1-\alpha}\right\} \right] \leq 4C_0,
\end{align*}
which indicates the following by applying Jensen's inequality:
\begin{align*}
     \E\left[ \frac{1}{T}\sum_{t=s}^T \Norm{\v_t}\right] \leq \max\{(4C_0)^{\frac{1}{2}},(4C_0)^{\frac{1}{2(1-\alpha)}}\} T^{-1/3}.
\end{align*}
Also, because of Jensen's inequality, we have:
\begin{align*}
    \mathbb{E}\left[\frac{1}{T}\sum_{t=1}^{s-1}||\v_t||\right]  &\leq \sqrt{\mathbb{E} \left[  \left(\frac{1}{T}\sum_{t=1}^{s-1}||\v_t||\right)^2 \right]} = \sqrt{\mathbb{E} \left[ \frac{1}{T^2} \left(\sum_{t=1}^{s-1}||\v_t||\right)^2 \right]} \\
    &\leq \sqrt{\mathbb{E} \left[ \frac{s}{T^2} \sum_{t=1}^{s-1}||\v_t||^2 \right]} \leq \sqrt{\mathbb{E} \left[ \frac{1}{T} \sum_{t=1}^{s-1}||\v_t||^2 \right]} \\
    &\leq  \sqrt{ \frac{1}{T} C_0 T^{1/3}} = \sqrt{C_0 }T^{-1/3}.
\end{align*}
Summing up, we have proven that 
\begin{align}\label{L2}
     \E\left[ \frac{1}{T}\sum_{t=1}^T \Norm{\v_t}\right] \leq \max\{3(C_0)^{\frac{1}{2}}, (C_0)^{\frac{1}{2}}+(4C_0)^{\frac{1}{2(1-\alpha)}}\} T^{-1/3}.
\end{align} 

Finally, we finish our proof by introducing the following lemma. 
\begin{lemma}\label{lem3} Suppose  $0<\beta<1$, our method ensures that
    \begin{align*}
      &\sum_{t=1}^{T} \E\left[\Norm{\v_t - \nabla f(\x_{t})}^2\right] 
        \leq  3\sigma^2 T^{1/3}  + \E\left[\frac{2L^2}{\beta}\sum_{t=1}^T \eta_t^2\Norm{\v_t}^2 \right].
    \end{align*}
\end{lemma}
\begin{proof}
    First note that we have already proven the following in equation~(\ref{lemma 6+}).
    \begin{align*}
    \E\left[\Norm{\nabla f(\x_{t+1}) -\v_{t+1}}^2\right] \leq  (1-\beta)\E\left[\|\v_{t} - \nabla f(\x_{t})\|^2\right]  + 2\beta^2\sigma^2 + 2L^2 \E\left[\eta_{t}^2\|\v_{t} \|^2\right] .
    \end{align*}

After rearranging the items, we can get the following:
    \begin{align*}
        \E\left[\Norm{\v_t - \nabla f(\x_{t})}^2\right]
        \leq &\frac{1}{\beta}\left(\E\left[\Norm{\v_t - \nabla f(\x_{t})}^2 \right]- \E\left[\Norm{\v_{t+1} - \nabla f(\x_{t+1})}^2\right] \right)+  2\beta  \sigma^2 \\
        &+ \frac{2L^2}{\beta}\E\left[\eta_t^2 \Norm{\v_t}^2\right].
    \end{align*}
    By summing up, we have:
    \begin{align}\label{lemma3-nobatch}
        \sum_{t=1}^{T} \E\left[\Norm{\v_t - \nabla f(\x_{t})}^2\right] 
        \leq  \frac{1}{\beta}\E\left[\Norm{\v_1 - \nabla f(\x_{1})}^2\right] 
        +  2\beta \sigma^2 T  + \frac{2L^2}{\beta}\sum_{t=1}^T\E\left[ \eta_t^2\Norm{\v_t}^2 \right].
    \end{align}
Since we use a large batch size in the first iteration, that is, $B_0 = T^{1/3}$, we can now ensure that $\E\left[\Norm{\v_1 - \nabla f(\x_{1})}^2\right] \leq \frac{\sigma^2}{B_0} = \frac{\sigma^2}{T^{1/3}}$. Note that $\beta=\frac{1}{T^{2/3}}$, so first term equals to $\sigma^2 T^{1/3}$ and the second term reduces to $2\sigma^2 T^{1/3}$. To this end, we ensure
 \begin{align*}
        \sum_{t=1}^{T} \E\left[\Norm{\v_t - \nabla f(\x_{t})}^2\right] 
        \leq  3\sigma^2 T^{1/3}  + \frac{2L^2}{\beta}\sum_{t=1}^T\E\left[ \eta_t^2\Norm{\v_t}^2 \right].
    \end{align*}
    Thus we finish the proof for this lemma.
\end{proof}
Here, we bound the term  $\frac{2L^2}{\beta}\sum_{t=1}^T \eta_t^2\Norm{\v_t}^2 $ as follows.
In the first stage, with $\eta_t =\frac{1}{T^{1/3}}$, we have:
\begin{align*}
    \frac{2L^2}{\beta}\sum_{t=1}^{s-1}\eta_t^2\Norm{\v_t}^2 = \frac{2L^2}{\beta}\sum_{t=1}^{s-1} \frac{1}{T^{2/3}}\Norm{\v_t}^2  \leq 2L^2 T^{1/3}.
\end{align*}
For the second stage, the analysis gives:
\begin{align*}
     \frac{2L^2}{\beta}\sum_{t=s}^T \eta_t^2\Norm{\v_t}^{2}
    \leq &2L^2T^{2/3}\sum_{t=s}^T\frac{\Norm{\v_t}^2 }{T^{\frac{2-2\alpha}{3}}(\sum_{i=s}^t\Norm{\v_i})^{2\alpha}}\\
    \leq & \frac{2L^2T^{2\alpha/3}}{1-2\alpha}\frac{((1-2\alpha)/L)^{1-2\alpha}}{((1-2\alpha)/L)^{1-2\alpha}}\left(\sum_{t=s}^T\Norm{\v_t}^2\right)^{1-2\alpha}\\
    \leq & 2\alpha \left(\frac{2L^2T^{2\alpha/3}((1-2\alpha)/L)^{1-2\alpha}}{1-2\alpha}\right)^{1/2\alpha}  + L\sum_{t=s}^T\Norm{\v_t}^2\\
    \leq &  2\alpha\left(\frac{2L^2((1-2\alpha)/L)^{1-2\alpha}}{1-2\alpha}\right)^{\frac{1}{2\alpha}}T^{1/3}  + L\sum_{t=1}^T\Norm{\v_t}^2,
\end{align*}
where the second inequality uses Lemma~\ref{lem1}, and the third one employs Young's inequality, such that $xy \leq 2\alpha x^{\frac{1}{2\alpha}} + (1-2\alpha)y^{\frac{1}{1-2\alpha}}$ for positive $x, y$. We also know that
\begin{align*}
     \frac{2L^2}{\beta}\sum_{t=s}^T \eta_t^2\Norm{\v_t}^{2}
    \leq &2L^2T^{2/3}\sum_{t=s}^T\frac{\Norm{\v_t}^2 }{T^{\frac{2-2\alpha}{3}}(\sum_{i=s}^t\Norm{\v_i})^{2\alpha}}\\
    \leq & \frac{2L^2T^{2\alpha/3}}{1-2\alpha}\left(\frac{1-2\alpha}{(1-\alpha)L}T^{-\frac{\alpha}{3}}\right)^{\frac{1-2\alpha}{1-\alpha}}\left(\frac{1-\alpha}{1-2\alpha}T^{\frac{\alpha}{3}}L\right)^{\frac{1-2\alpha}{1-\alpha}}\left(\sum_{t=s}^T\Norm{\v_t}^2\right)^{1-2\alpha}\\
    \leq & \frac{\alpha}{1-\alpha} \left(\frac{2L^2T^{2\alpha/3}}{1-2\alpha}\left(\frac{1-2\alpha}{(1-\alpha)L}T^{-\frac{\alpha}{3}}\right)^{\frac{1-2\alpha}{1-\alpha}}\right)^{\frac{1-\alpha}{\alpha}}  + T^{\frac{\alpha}{3}}L\left(\sum_{t=s}^T\Norm{\v_t}^2\right)^{1-\alpha}\\
    \leq &  \frac{\alpha}{1-\alpha} \left(\frac{2L^2}{1-2\alpha}\left(\frac{1-2\alpha}{(1-\alpha)L}\right)^{\frac{1-2\alpha}{1-\alpha}}\right)^{\frac{1-\alpha}{\alpha}}T^{1/3}  + T^{\frac{\alpha}{3}}L\left(\sum_{t=s}^T\Norm{\v_t}^2\right)^{1-\alpha},
\end{align*}
where the second inequality uses Lemma~\ref{lem1}, and the third one employs Young's inequality, such that $xy \leq \frac{\alpha}{1-\alpha} x^{\frac{1-\alpha}{\alpha}} + \frac{1-2\alpha}{1-\alpha}y^{\frac{1-\alpha}{1-2\alpha}}$ for positive $x, y$.

As a result, we have:
\begin{align*}
    &\sum_{t=1}^{T} \E\left[\Norm{\v_t - \nabla f(\x_{t})}^2\right]        \\ \leq&  \left(3\sigma^2+4C_0L+
    L^{\frac{1+2\alpha}{2\alpha}}\left(\frac{2(1-2\alpha)^{1-2\alpha}}{1-2\alpha}\right)^{\frac{1}{2\alpha}}+ L^{\frac{1}{\alpha}} \left(\frac{2}{1-2\alpha}\left(\frac{1-2\alpha}{1-\alpha}\right)^{\frac{1-2\alpha}{1-\alpha}}\right)^{\frac{1-\alpha}{\alpha}}\right) T^{1/3} .
\end{align*}

By integrating these findings, we can finally have:
\begin{align*}
    \E\left[\frac{1}{T}\sum_{t=1}^T  \Norm{\nabla f(\x_t)}\right]
    \leq  \frac{1}{T}\E\left[\sum_{t=1}^T  \Norm{\v_t} \right]+  \frac{1}{T}\left[\sum_{t=1}^T\Norm{\nabla f(\x_t) -\v_t} \right] 
    \leq  \frac{C^{\prime}}{T^{1/3}}, 
\end{align*}

where 
\begin{equation*}
    \begin{split}
        C^{\prime} &=\max\{3(C_0)^{\frac{1}{2}}, (C_0)^{\frac{1}{2}}+(4C_0)^{\frac{1}{2(1-\alpha)}}\}\\
        &\quad+\sqrt{3\sigma^2+4C_0L+
    L^{\frac{1+2\alpha}{2\alpha}}\left(\frac{2(1-2\alpha)^{1-2\alpha}}{1-2\alpha}\right)^{\frac{1}{2\alpha}}+ L^{\frac{1}{\alpha}} \left(\frac{2}{1-2\alpha}\left(\frac{1-2\alpha}{1-\alpha}\right)^{\frac{1-2\alpha}{1-\alpha}}\right)^{\frac{1-\alpha}{\alpha}}}\\
        &=\mathcal{O}\left(\Delta_f^{\frac{1}{2(1-\alpha)}}+\sigma^{\frac{1}{1-\alpha}} + L^{\frac{1}{2\alpha}} \right),
    \end{split}
\end{equation*}
with $C_0$ defined in equation~(\ref{C0}).
We find that larger $\alpha$ leads to better dependence on $L$ and worse reliance on parameters $\Delta$ and $\sigma$. For $\alpha \to \frac{1}{3}$, we can obtain that 
\begin{align*}
    &\E\left[\frac{1}{T}\sum_{t=1}^T  \Norm{\nabla f(\x_t)}\right] \leq \mathcal{O}\left(\frac{\Delta_f^{3/4}+\sigma^{3/2} + L^{3/2}}{T^{1/3}}\right).
\end{align*}
Since we require $0<\alpha < \frac{1}{3}$, in practice, we can use $\alpha = 0.3$ instead, which leads to a convergence rate of $\mathcal{O}\left(\frac{\Delta_f^{5/7}+\sigma^{10/7} + L^{{5}/{3}}}{T^{1/3}}\right)$.

\section{Proof of Theorem~\ref{thm:main_1}}
Since $2^0+2^1+\cdots+2^{K-1} < 2^K$, running the algorithm for $T$ iterations guarantees at least $K = \lfloor \log (T) \rfloor $ complete stages. In the theoretical analysis, we can simply use the output of the last complete stage $K=\lfloor \log (T) \rfloor$, which has been at least run for $2^{K-1} \geq T/4$ iterations. According to the analysis of Theorem~\ref{thm:main_0}, we have already known that running the Algorithm~\ref{alg:storm} for $T/4$ iterations leads to the following guarantee:
\begin{align*}
    \E\left[ \Norm{\nabla f(\x_{\tau})} \right] \leq \mathcal{O}\left(\frac{\Delta_f^{\frac{1}{2(1-\alpha)}}+\sigma^{\frac{1}{1-\alpha}} + L^{\frac{1}{2\alpha}}}{{(T/4)}^{1/3}}\right)=\mathcal{O}\left(\frac{\Delta_f^{\frac{1}{2(1-\alpha)}}+\sigma^{\frac{1}{1-\alpha}} + L^{\frac{1}{2\alpha}}}{{T}^{1/3}}\right),
\end{align*}
which is on the same order of the original convergence rate.

\section{Proof of Theorem~\ref{thm:main_0+}}
According to equation~(\ref{smooth}), we have already proven that
\begin{align*}
  \sum_{t=1}^T \eta_t \Norm{\v_t}^2 \leq 2F(\x_1) - 2F(\x_{T+1}) + \sum_{t=1}^T \eta_t \Norm{\nabla F(\x_t) - \v_t}^2 + \sum_{t=1}^T \eta_t^2 L \Norm{\v_t}^2.  
\end{align*}
Then we bound the term  $\sum_{t=1}^T \eta_t \Norm{\nabla F(\x_t) - \v_t}^2$ as follows:
\begin{align*}
\begin{split}
     \Norm{\nabla F(\x_t) - \v_t}^2
    & \leq 2 \Norm{\nabla f(g(\x_t))\nabla g(\x_t) - \nabla f(\u_t) \nabla g(\x_t)}^2 + 2 \Norm{\nabla f(\u_t) \nabla g(\x_t) -\v_t}^2\\
    & \leq 2C^2L^2\Norm{g(\x_t) - \u_t}^2+ 2\Norm{\v_t - \nabla f(\u_t) \nabla g(\x_t)}^2.
\end{split}
\end{align*}
Define that $G_t = \nabla f(\u_t) \nabla g(\x_t)$, then we have:
\begin{align*}
    \sum_{t=1}^T  \E\left[\eta_t\Norm{\nabla F(\x_t) - \v_t}^2\right] \leq 2C^2L^2\sum_{t=1}^T  \E\left[\eta_t\Norm{g(\x_t) - \u_t}^2\right] + 2\sum_{t=1}^T \E\left[ \eta_t\Norm{\v_t - G_t}^2\right].
\end{align*}

For the term $\sum_{t=1}^T \E\left[\eta_t \Norm{\v_t - G_t}^2\right]$, following the very similar analysis of  equation~(\ref{lemma 6+}), we have the following guarantee:
\begin{align*}
    &\quad \E_{\xi_{t+1},\zeta_{t+1}}\left[\Norm{\v_{t+1} - G_{t+1}}^2\right] \\
    & \leq (1-\beta)\Norm{\v_{t} - G_{t}}^2 \\
    &\qquad + 2\beta^2 \E_{\xi_{t+1},\zeta_{t+1}}\left[\Norm{\nabla f(\u_{t+1};\xi_{t+1})\nabla g(\x_{t+1};\zeta_{t+1}) - \nabla f(\u_{t+1}) \nabla g(\x_{t+1})}^2\right]  \\
    & \qquad + 2 \E_{\xi_{t+1},\zeta_{t+1}}\left[\Norm{\nabla f(\u_{t+1};\xi_{t+1})\nabla g(\x_{t+1};\zeta_{t+1}) - \nabla f(\u_{t};\xi_{t+1})\nabla g(\x_{t};\zeta_{t+1})}^2\right]\\
    & \leq (1-\beta)\Norm{\v_{t} - G_{t}}^2+ 4C^2\sigma^2\beta^2 + 4C^2L^2  \E\left[\eta_t^2 \Norm{\v_t}^2\right] + 4C^2L^2 \E\left[\Norm{\u_{t+1}-\u_{t}}^2\right].
\end{align*}

That is to say:
\begin{align*}
    &\sum_{t=1}^T  \E\left[\eta_t\Norm{\v_{t} - G_{t}}^2\right] \\
    \leq& \E\left[\frac{\eta_1}{\beta}\Norm{\v_{1} - G_{1}}^2\right] + 4C^2\sigma^2\beta\E\left[\sum_{t=1}^T\eta_t \right]+ \frac{4C^2L^2}{\beta}\sum_{t=1}^T \E\left[\eta_t^3  \Norm{\v_t}^2\right] \\
    &\quad + \frac{4C^2L^2}{\beta} \sum_{t=1}^T \E\left[\eta_t \Norm{\u_{t+1}-\u_{t}}^2\right] \\
    \leq& (1+4C^2)\sigma^2 + \frac{4C^2L^2}{\beta}\sum_{t=1}^T  \E\left[ \eta_t^3\Norm{\v_t}^2\right] + \frac{4C^2L^2}{\beta} \sum_{t=1}^T  \E\left[\eta_t\Norm{\u_{t+1}-\u_{t}}^2\right],
\end{align*}
where the last inequality due to the fact that $\beta=T^{-2/3}$, $\eta_t \leq T^{-1/3}$, and we use a large batch size $T^{1/3}$ in the first iteration. Next, we further bound the term $\sum_{t=1}^T  \E\left[\eta_{t-1}\Norm{\u_t-\u_{t-1}}^2\right].$ First, we can ensure that:

\begin{align*}
   &\quad \E_{\zeta_t}\left[\Norm{\u_t-\u_{t-1}}^2\right] \\&  =\E_{\zeta_t}\left[\Norm{\beta(g(\x_t;\zeta_t) - \u_{t-1}) + (1-\beta)(g(\x_t;\zeta_t)  - g(\x_{t-1};\zeta_t) )}^2\right] \\
   & = \E_{\zeta_t}\left[\Norm{\beta(g(\x_{t-1}) - \u_{t-1}) + (g(\x_t;\zeta_t)  - g(\x_{t-1};\zeta_t) ) + \beta (g(\x_{t-1};\zeta_t)  - g(\x_{t-1}))}^2\right]\\
   & \leq 3\beta^2 \Norm{g(\x_{t-1}) -\u_{t-1}}^2 + 3C^2\eta_{t-1}^2\Norm{\v_{t-1}}^2 +3\beta^2\sigma^2.
\end{align*}

So we know that \begin{align*}
    &\frac{1}{\beta} \sum_{t=1}^T  \E\left[\eta_t\Norm{\u_{t+1}-\u_{t}}^2\right] \\
    \leq &3\beta \sum_{t=1}^T  \E\left[\eta_t\Norm{g(\x_{t}) -\u_{t}}^2\right] + 3C^2  \E\left[\sum_{t=1}^T\frac{\eta_t^3}{\beta}\Norm{\v_t}^2\right] +3\beta\sigma^2\E\left[\sum_{t=1}^T \eta_t\right].
\end{align*}

So far, we have 
\begin{align*}
    \sum_{t=1}^T  \E\left[\eta_t\Norm{\nabla F(\x_t) - \v_t}^2\right] \leq & 26C^2L^2 \sum_{t=1}^T  \E\left[\eta_t\Norm{\u_t - g(\x_t)}^2\right]  \\
         + &\frac{8C^2L^2 +24C^4L^2}{\beta}\sum_{t=1}^T\E\left[\eta_t^3\Norm{\v_t}^2\right] + (2+8C^2+24C^2L^2)\sigma^2.
\end{align*}

Next, we can bound $\sum_{t=1}^T \eta_t \Norm{\u_t - g(\x_t)}^2$ following equation~(\ref{lemma 6+}), as:
\begin{align*}
    \Norm{\u_t - g(\x_t)}^2 \leq \frac{1}{\beta}\left(\Norm{\u_t - g(\x_t)}^2 - \E_{\zeta_{t+1}}\left[\Norm{\u_{t+1} - g(\x_{t+1})}^2\right]\right) + 2\beta\sigma^2 + \frac{2C^2\eta_t^2}{\beta}\Norm{\v_t}^2.
\end{align*}

So we can have
\begin{align*}
    \E\left[\sum_{t=1}^T \eta_t\Norm{\u_t - g(\x_t)}^2 \right]&\leq \E\left[\frac{\eta_1}{\beta}\Norm{\u_1 - g(\x_1)}^2\right] + 2\beta\sigma^2 \E\left[\sum_{t=1}^T\eta_t \right]+ \frac{2C^2}{\beta}\sum_{t=1}^T \E\left[\eta_t^3\Norm{\v_t}^2\right]\\
    &\leq 3\sigma^2+ \frac{2C^2}{\beta}\sum_{t=1}^T \E\left[\eta_t^3\Norm{\v_t}^2\right],
\end{align*}
where the last inequality due to the fact that $\beta=T^{-2/3}$, $\eta_t \leq T^{-1/3}$, and we use a large batch size $T^{1/3}$ in the first iteration.
Combining all, we have:
\begin{align*}
    \sum_{t=1}^T  \E\left[\eta_t\Norm{\v_t}^2\right] &\leq  2\Delta_F    +(2+8C^2+102C^2L^2)\sigma^2  \\
    &\quad +  (8C^2L^2+76C^4L^2   ) \E\left[\sum_{t=1}^T \frac{\eta_t^3}{\beta} \Norm{\v_t}^2\right] + L \E\left[\sum_{t=1}^T \eta_t^2 \Norm{\v_t}^2\right].
\end{align*}
Treating $\Delta_F,C,L,\sigma$ as constant, the above inequality is very similar to Lemma~\ref{lem2}.
Thus following the very similar analysis after Lemma~\ref{lem2}, we can show that:
\begin{align*}
    \E\left[\frac{1}{T}\sum_{t=1}^T \Norm{\v_t}\right] \leq \mathcal{O}(T^{-1/3}).
\end{align*}
According to previous analysis, we also have that
\begin{align*}
    &\sum_{t=1}^T  \E\left[\Norm{\nabla f(\x_t) - \v_t}^2\right] \\
    \leq & 26C^2L^2 \sum_{t=1}^T  \E\left[\Norm{\u_t - g(\x_t)}^2\right] + \frac{8C^2L^2 +24C^4L^2}{\beta}\E\left[\sum_{t=1}^T\eta_t^2\Norm{\v_t}^2\right] \\
    &\qquad + (2+8C^2+24C^2L^2)\sigma^2 T^{1/3}\\
    \leq &  (2+8C^2+102C^2L^2)\sigma^2T^{1/3}  + {8C^2L^2 +76C^4L^2}\E\left[\sum_{t=1}^T\frac{\eta_t^2}{\beta}\Norm{\v_t}^2\right], 
\end{align*}
which is similar to Lemma~\ref{lem3}, and leads to $\sum_{t=1}^T  \E\left[\Norm{\nabla f(\x_t) - \v_t}/T\right] \leq \mathcal{O}(T^{-1/3})$ following the same analysis.
Combing all these together, we can deduce that
\begin{align*}
    \frac{1}{T}\sum_{t=1}^T \Norm{\nabla F(\x_t)} \leq \mathcal{O}(T^{-1/3}),
\end{align*}
which finishes the proof of Theorem~\ref{thm:main_0+}.

\section{Proof of Theorem~\ref{thm:main_2}}
Due to the smoothness of $F(\x)$, we have proven the following in equation~(\ref{smooth}):
\begin{align*}
    \sum_{t=1}^T {\eta_t}\Norm{\v_t}^2  \leq 2F(\x_1) -2F(\x_{T+1})+ \sum_{t=1}^T {\eta_t}\Norm{\nabla F(\x_t) -\v_t}^2   +\sum_{t=1}^T{\eta_t^2 L} \Norm{\v_t}^2.
\end{align*}

For the LHS, we have the following guarantee:
    \begin{align*}
    \sum_{t=1}^T {\eta_t}\Norm{\v_t}^2& =   \sum_{t=1}^T\frac{\Norm{\v_t}^2}{n^{\frac{1-\alpha}{2}}\left( \sum_{i=1}^t \Norm{\v_i}^2\right)^\alpha} \geq \left(\frac{1}{\sqrt{n}}\sum_{t=1}^T \Norm{\v_t}^2\right)^{1-\alpha} .
\end{align*}

Then, we bound the terms in the RHS. First, we have the following lemma.
\begin{lemma} Define that $ \z_{t} = \nabla f_{i_t}(\x_t)  - g_t^{i_t} +\frac{1}{n} \sum_{i=1}^n g_t^i$, we have:
    \begin{align*}
    \E\left[\sum_{t=1}^T \eta_t \Norm{\nabla F(\x_t) -\v_t}^2\right] \leq 2\beta \E\left[\sum_{t=1}^T \eta_t \Norm{\nabla F(\x_{t+1}) -\z_{t+1}}^2\right] + 2L^2 \E\left[\sum_{t=1}^T \frac{\eta_t^3}{\beta} \Norm{\v_t}^2\right]. 
    \end{align*}
\end{lemma}
\begin{proof} According to the definition of $\z_t$, the estimator $\v_t$ can be expressed as 
\begin{align*}
    \v_t =(1-\beta) \v_{t-1}+ \beta \z_t +(1-\beta)\left(\nabla f_{i_t}(\x_t) - \nabla f_{i_t}(\x_{t-1})\right).
\end{align*}

Note that $\E_{i_{t+1}}\left[ \z_{t+1}\right] = \nabla F(\x_{t+1})$, and we have:
    \begin{align}\label{SSA}
    \begin{split}
        &\quad\ \E_{i_{t+1}}\left[\Norm{\nabla F(\x_{t+1}) -\v_{t+1}}^2\right]\\
        & = \E_{i_{t+1}}\left[\Norm{(1-\beta)\v_{t} + \beta \z_{t+1} + (1-\beta)\left(\nabla f_{i_{t+1}}(\x_{t+1}) - \nabla f_{i_{t+1}}(\x_{t}) \right) - \nabla F(\x_{t+1})}^2\right]\\
        & = \E_{i_{t+1}}\left[\|(1-\beta)(\v_t - \nabla F(\x_{t})) + \beta(\z_{t+1}-\nabla F(\x_{t+1})) \right. \\
        & \quad \left.+ (1-\beta)\left(\nabla f_{i_{t+1}}(\x_{t+1}) - \nabla f_{i_{t+1}}(\x_{t}) + \nabla F(\x_t) - \nabla F(\x_{t+1})  \right) \|^2 \right]\\
         & \leq  \Norm{(1-\beta)(\v_t - \nabla F(\x_{t}))}^2 + \E_{i_{t+1}}\left[ \left\| \beta(\z_{t+1}-\nabla F(\x_{t+1})) \right.\right.\\
         &\left.\left. \quad  
         + (1-\beta)\left(\nabla f_{i_{t+1}}(\x_{t+1}) - \nabla f_{i_{t+1}}(\x_{t}) + \nabla F(\x_t) - \nabla F(\x_{t+1})  \right)\right\|^2 \right]\\
         & \leq (1-\beta)^2 \Norm{\v_t - \nabla F(\x_{t})}^2 + 2\beta^2 \E_{i_{t+1}}\left[\Norm{\z_{t+1}-\nabla F(\x_{t+1})}^2\right]\\
         & \quad+ 2(1-\beta)^2\E_{i_{t+1}}\left[\Norm{\nabla f_{i_{t+1}}(\x_{t+1}) - \nabla f_{i_{t+1}}(\x_{t}) + \nabla F(\x_t) - \nabla F(\x_{t+1})}^2 \right]\\
         & \leq  (1-\beta)^2 \Norm{\v_t - \nabla F(\x_{t})}^2 + 2\beta^2 \E_{i_{t+1}}\left[\Norm{\z_{t+1}-\nabla F(\x_{t+1})}^2\right]\\
         &\quad + 2(1-\beta)^2 \E_{i_{t+1}}\left[\Norm{\nabla f_{i_{t+1}}(\x_{t+1}) - \nabla f_{i_{t+1}}(\x_{t})}^2 \right]\\
          & \leq  (1-\beta)^2 \Norm{\v_t - \nabla F(\x_{t})}^2 + 2\beta^2 \E_{i_{t+1}}\left[\Norm{\z_{t+1}-\nabla F(\x_{t+1})}^2\right]+ 2L^2 \Norm{\x_{t+1} -\x_{t}}^2 \\
         & \leq  (1-\beta) \Norm{\v_t - \nabla F(\x_{t})}^2 + 2\beta^2 \E_{i_{t+1}}\left[\Norm{\z_{t+1}-\nabla F(\x_{t+1})}^2\right] + 2L^2\eta_t^2 \Norm{\v_t}^2.
\end{split}
    \end{align}
    Rearrange the items and multiply the both sides by $\eta_t$, we can get the following:
    \begin{align*}
          \E\left[\eta_t\Norm{\v_t - \nabla F(\x_{t})}^2\right] \leq & \E\left[\frac{\eta_t}{\beta}\Norm{\v_t - \nabla F(\x_{t})}^2\right] - \E\left[\frac{\eta_t}{\beta}\Norm{\v_{t+1} - \nabla F(\x_{t+1})}^2\right] \\
         &\quad + 2\beta  \E\left[\eta_t \Norm{\z_{t+1}-\nabla F(\x_{t+1})}^2\right] + \frac{2L^2}{\beta}\E\left[\eta_t^3 \Norm{\v_t}^2 \right].
    \end{align*}
    Note that $\eta_t$ is non-increasing. By summing up, we have:
    \begin{align*}
        &\sum_{t=1}^{T} \E\left[\eta_t\Norm{\v_t - \nabla F(\x_{t})}^2\right] \\
        \leq & \E\left[\frac{\eta_1}{\beta}\Norm{\v_1 - \nabla F(\x_{1})}^2\right] +  2\beta\sum_{t=1}^T \E\left[\eta_t\Norm{\z_{t+1}-\nabla F(\x_{t+1})}^2\right]   + \frac{2\LL^2}{\beta}\sum_{t=1}^T\E\left[\eta_t^3 \Norm{\v_t}^2 \right].
    \end{align*}
Since we use a full batch in the first iteration, we can    finish the proof of this lemma.
\end{proof}
Next, we bound two terms in the above lemma.
\begin{lemma}\label{SAG} We can ensure that
    \begin{align*}
    2\beta \sum_{t=1}^T  \E\left[\eta_t\Norm{\nabla F(\x_{t+1}) -\z_{t+1}}^2\right] \leq 12 \LL^2\sum_{t=1}^T\E\left[\frac{\eta_t^3}{\beta}\Norm{\v_t}^2\right].
    \end{align*}
\end{lemma}
\begin{proof}
    \begin{align}\label{fini}
    \begin{split}
        & \E_{i_{t+1}}\left[\eta_t\Norm{\nabla F(\x_{t+1}) -\z_{t+1}}^2\right] \\
        =&  \E_{i_{t+1}}\left[\eta_t\Norm{\nabla F(\x_{t+1}) -\nabla f_{i_{t+1}}(\x_{t+1}) + g_{t+1}^{i_{t+1}} - \frac{1}{n}\sum_{i=1}^n g_{t+1}^i}^2\right]\\
        =&  \E_{i_{t+1}}\left[\eta_t\Norm{\nabla f_{i_{t+1}}(\x_{t+1}) - g_{t+1}^{i_{t+1}} - \left( \nabla F(\x_{t+1}) - \frac{1}{n}\sum_{i=1}^n g_{t+1}^i\right)}^2\right]\\
        \leq &  \E_{i_{t+1}}\left[\eta_t\Norm{\nabla f_{i_{t+1}}(\x_{t+1}) - g_{t+1}^{i_{t+1}}}^2\right] \\
        = &  \frac{1}{n}\sum_{i=1}^{n} \eta_t\Norm{\nabla f_{i}(\x_{t+1}) - g_{t+1}^{i}}^2,
    \end{split}
    \end{align}
where the last equation is due to the fact that $i_{t+1}$ is randomly sample from $\{ 1,2,\cdots,n\}$.
Note that we also have that 
\begin{align*}
&\frac{1}{n}\sum_{i=1}^{n} \eta_t\Norm{\nabla f_{i}(\x_{t+1}) - g_{t+1}^{i}}^2=\E_{i_{t+1}}\left[\eta_t\Norm{\nabla f_{i_{t+1}}(\x_{t+1}) - g_{t+1}^{i_{t+1}}}^2\right] \\
          \leq &\E_{i_{t+1}}\left[\eta_t(1+2n)\Norm{\nabla f_{i_{t+1}}(\x_{t+1}) - \nabla f_{i_{t+1}}(\x_{t})}^2 + \eta_t(1+\frac{1}{2n})\Norm{\nabla f_{i_{t+1}}(\x_{t}) - g_{t+1}^{i_{t+1}}}^2\right]\\
        \leq & \E_{i_{t+1}}\left[(1+2n)\LL^2\eta_t^3\Norm{\v_t}^2 + \eta_t(1+\frac{1}{2n})\Norm{\nabla f_{i_{t+1}}(\x_{t}) - g_{t+1}^{i_{t+1}}}^2\right]\\
        \leq & \E_{i_{t+1}}\left[ (1+2n)\LL^2\eta_t^3\Norm{\v_t}^2 \right.\\
        &\quad \left.+ \eta_t(1+\frac{1}{2n})\left((1-\frac{1}{n})\Norm{\nabla f_{i_{t+1}}(\x_{t}) - g_{t}^{i_{t+1}}}^2+ \frac{1}{n}\Norm{\nabla f_{i_{t+1}}(\x_{t}) - \nabla f_{i_{t+1}}(\x_{t})}^2\right)\right]\\
        \leq & \E_{i_{t+1}}\left[3n\LL^2\eta_t^3\Norm{\v_t}^2 + \eta_t(1-\frac{1}{2n})\Norm{\nabla f_{i_{t+1}}(\x_{t}) - g_{t}^{i_{t+1}}}^2\right] \\
        \leq & 3n\LL^2\eta_t^3\Norm{\v_t}^2 + \left(1-\frac{1}{2n}\right)\frac{1}{n}\sum_{i=1}^n \eta_t\Norm{\nabla f_{i}(\x_{t}) - g_{t}^{i}}^2.
\end{align*}
That is to say, we can ensure that
\begin{align*}
\frac{1}{n}\sum_{i=1}^{n} \eta_{t+1}\Norm{\nabla f_{i}(\x_{t+1}) - g_{t+1}^{i}}^2
        \leq &\sum_{i=1}^{n} \eta_t\Norm{\nabla f_{i}(\x_{t+1}) - g_{t+1}^{i}}^2\\
        \leq &3n\LL^2\eta_t^3\Norm{\v_t}^2 + \left(1-\frac{1}{2n}\right)\frac{1}{n}\sum_{i=1}^n \eta_t\Norm{\nabla f_{i}(\x_{t}) - g_{t}^{i}}^2.
\end{align*}
By rearranging and summing up, we have
    \begin{align*}
&\frac{1}{2n}\sum_{t=1}^{T}\frac{1}{n}\sum_{i=1}^{n} \eta_t\Norm{\nabla f_{i}(\x_{t}) - g_{t}^{i}}^2
        \leq 3n\LL^2\sum_{t=1}^{T} \eta_t^3\Norm{\v_t}^2 +\frac{1}{n}\sum_{i=1}^n \eta_1 \Norm{\nabla f_{i}(\x_{1}) - g_{1}^{i}}^2.
\end{align*}
Since we use a full batch $n$ in the first iteration, the second term equals zero, and thus we obtain:
 \begin{align*}
&\sum_{t=1}^{T}\frac{1}{n}\sum_{i=1}^{n} \eta_t\Norm{\nabla f_{i}(\x_{t}) - g_{t}^{i}}^2
        \leq 6n^2\LL^2\sum_{t=1}^{T} \eta_t^3\Norm{\v_t}^2= \frac{6\LL^2}{\beta^2}\sum_{t=1}^{T} \eta_t^3\Norm{\v_t}^2,
\end{align*}
which leads to the result of this lemma.
\end{proof}

\begin{lemma}We have the following guarantee:
    \begin{align*}
    2\LL^2\sum_{t=1}^T \frac{\eta_t^3}{\beta} \Norm{\v_t}^2 \leq \frac{2\alpha}{1-\alpha}\left(\frac{2\LL^2}{1-3\alpha} \left(\frac{24-72\alpha}{1-\alpha}\right)^{\frac{1-3\alpha}{1-\alpha}} \right)^{\frac{1-\alpha}{2\alpha}} + \frac{1}{24}\left(\frac{1}{\sqrt{n}}\sum_{t=1}^T\Norm{\v_t}^2\right)^{1-\alpha}.
    \end{align*}
\end{lemma}
\begin{proof} 
    \begin{align*}
        &2\LL^2 \sum_{t=1}^T \frac{\eta_t^3}{\beta} \Norm{\v_t}^2\\
        = & 2\LL^2n \sum_{t=1}^T \eta_t^3 \Norm{\v_t}^2
        = 2\LL^2 \sum_{t=1}^T \frac{\Norm{\v_t}^2}{n^{\frac{1-3\alpha}{2}}\left(\sum_{i=1}^t \Norm{\v_i}^2\right)^{3\alpha}}\\
        \leq &  \frac{2\LL^2}{1-3\alpha}\left(\frac{1}{\sqrt{n}}\sum_{t=1}^T\Norm{\v_t}^2\right)^{1-3\alpha}\\
        =& \frac{2\LL^2}{1-3\alpha} \left(\frac{24-72\alpha}{1-\alpha}\right)^{\frac{1-3\alpha}{1-\alpha}}\left(\frac{1-\alpha}{24-72\alpha}\right)^{\frac{1-3\alpha}{1-\alpha}}\left(\frac{1}{\sqrt{n}}\sum_{t=1}^T\Norm{\v_t}^2\right)^{1-3\alpha}\\
        \leq & \frac{2\alpha}{1-\alpha}\left(\frac{2\LL^2}{1-3\alpha} \left(\frac{24-72\alpha}{1-\alpha}\right)^{\frac{1-3\alpha}{1-\alpha}} \right)^{\frac{1-\alpha}{2\alpha}}+ \frac{1}{24}\left(\frac{1}{\sqrt{n}}\sum_{t=1}^T\Norm{\v_t}^2\right)^{1-\alpha},
    \end{align*}
where the last inequality employs Young's inequality, such that $xy \leq \frac{2\alpha}{1-\alpha} x^{\frac{1-\alpha}{2\alpha}} + \frac{1-3\alpha}{1-\alpha} y^{\frac{1-\alpha}{1-3\alpha}}$ for positive $x, y$. 
\end{proof}

\begin{lemma} We can ensure the following guarantee:
    \begin{align*}
    \E\left[\sum_{t=1}^T{\eta_t^2 \LL}\Norm{\v_t}^2\right] \leq \frac{\alpha}{1-\alpha}\left(\frac{\LL}{1-2\alpha} \left(\frac{4-8\alpha}{1-\alpha}\right)^{\frac{1-2\alpha}{1-\alpha}}\right)^{\frac{1-\alpha}{\alpha}} + \frac{1}{4}\E\left[\left(\frac{1}{\sqrt{n}}\sum_{t=1}^T\Norm{\v_t}^2\right)^{1-\alpha}\right]
    \end{align*}
\end{lemma}
\begin{proof}
    \begin{align*}
        \sum_{t=1}^T\eta_t^2 \LL\Norm{\v_t}^2
        = & {\LL} \sum_{t=1}^T\frac{\Norm{\v_t}^2}{n^{1-\alpha}\left(\sum_{i=1}^t \Norm{\v_i}^2\right)^{2\alpha}}\\
        \leq & \frac{\LL}{1-2\alpha}\frac{1}{n^{1-\alpha}}\left(\sum_{t=1}^T\Norm{\v_t}^2\right)^{1-2\alpha}\\
        \leq & \frac{\LL}{1-2\alpha}\left(\frac{1}{\sqrt{n}}\sum_{t=1}^T\Norm{\v_t}^2\right)^{1-2\alpha}\\
        =& \frac{\LL}{1-2\alpha} \left(\frac{4-8\alpha}{1-\alpha}\right)^{\frac{1-2\alpha}{1-\alpha}}\left(\frac{1-\alpha}{4-8\alpha}\right)^{\frac{1-2\alpha}{1-\alpha}}\left(\frac{1}{\sqrt{n}}\sum_{t=1}^T\Norm{\v_t}^2\right)^{1-2\alpha}\\
        \leq& \frac{\alpha}{1-\alpha}\left(\frac{\LL}{1-2\alpha} \left(\frac{4-8\alpha}{1-\alpha}\right)^{\frac{1-2\alpha}{1-\alpha}}\right)^{\frac{1-\alpha}{\alpha}}+\frac{1}{4}\left(\frac{1}{\sqrt{n}}\sum_{t=1}^T\Norm{\v_t}^2\right)^{1-\alpha},
    \end{align*}
\end{proof}
where the last inequality employs Young's inequality, such that $xy \leq \frac{\alpha}{1-\alpha} x^{\frac{1-\alpha}{\alpha}} + \frac{1-2\alpha}{1-\alpha} y^{\frac{1-\alpha}{1-2\alpha}}$ for positive $x, y$.
Combing all these, we have already proven that 
\begin{align*}
    \E\left[\sum_{t=1}^T {\eta_t}\Norm{\v_t}^2\right] \geq \E\left[\left(\frac{1}{\sqrt{n}}\sum_{t=1}^T \Norm{\v_t}^2\right)^{1-\alpha} \right],\\
    \sum_{t=1}^T {\eta_t}\Norm{\v_t}^2  \leq C_3+\frac{3}{4}\E\left[\left(\frac{1}{\sqrt{n}}\sum_{t=1}^T\Norm{\v_t}^2\right)^{1-\alpha}\right],
\end{align*}
where 
\begin{align*}
    C_3 = &2\Delta_F +\frac{14\alpha}{1-\alpha}\left(\frac{2\LL^2}{1-3\alpha} \left(\frac{24-72\alpha}{1-\alpha}\right)^{\frac{1-3\alpha}{1-\alpha}} \right)^{\frac{1-\alpha}{2\alpha}}\\
    &+ \frac{\alpha}{1-\alpha}\left(\frac{\LL}{1-2\alpha} \left(\frac{4-8\alpha}{1-\alpha}\right)^{\frac{1-2\alpha}{1-\alpha}}\right)^{\frac{1-\alpha}{\alpha}}.
\end{align*}

So we can know that:
\begin{align*}
    \E\left[\left(\frac{1}{\sqrt{n}}\sum_{t=1}^T\Norm{\v_t}^2\right)^{1-\alpha}\right] \leq 4C_3.
\end{align*}

which indicate that
\begin{align*}
    \E\left[\left(\frac{1}{T}\sum_{t=1}^T\Norm{\v_t}^2\right)^{1-\alpha}\right] \leq 4C_3 \left(\frac{\sqrt{n}}{T} \right)^{1-\alpha}.
\end{align*}
as well as
\begin{align*}
\E\left[\frac{1}{T}\sum_{t=1}^T\Norm{\v_t}\right] \leq (4C_3)^{\frac{1}{2(1-\alpha)}} \cdot \frac{n^{1/4}}{T^{1/2}}. 
\end{align*}
To finish the proof, we also have to show the following lemma.
\begin{lemma}
    \begin{align*}
    &\E\left[\sum_{t=1}^T \Norm{\nabla F(\x_t) -\v_t}^2\right] \\
    \leq &\frac{\alpha \sqrt{n}}{1-\alpha} \left(\frac{14\LL^2 }{1-2\alpha} \left(\frac{2-4\alpha}{(1-\alpha)L}\right)^{\frac{1-2\alpha}{1-\alpha}}\right)^{\frac{1-\alpha}{\alpha}}+ \frac{n^{\frac{\alpha}{2}}L}{2}\E\left[\left(\sum_{t=1}^T\Norm{\v_t}^2\right)^{1-\alpha}\right].
\end{align*}
\end{lemma}
\begin{proof}
    According to the previous proof, we know that:
    \begin{align*}
        &\sum_{t=1}^T \E\left[\Norm{\nabla F(\x_t) -\v_t}^2\right]\\ \leq &2\beta \sum_{t=1}^T \E\left[ \Norm{\nabla F(\x_{t+1}) -\z_{t+1}}^2\right] + 2\LL^2 \sum_{t=1}^T \E\left[\frac{\eta_t^2}{\beta}\Norm{\v_t}^2\right] \leq 14n\LL^2 \sum_{t=1}^T\E\left[\eta_t^2\Norm{\v_t}^2\right].
    \end{align*}
    Also, we can deduce that:
    \begin{align*}
        14n\LL^2 \sum_{t=1}^T\eta_t^2\Norm{\v_t}^2
        = & 14\LL^2 n^\alpha \sum_{t=1}^T\frac{\Norm{\v_t}^2}{\left(\sum_{i=1}^t\Norm{\v_i}^2\right)^{2\alpha}}\\
        \leq & \frac{14\LL^2 n^\alpha}{1-2\alpha}\left(\sum_{t=1}^T\Norm{\v_t}^2\right)^{1-2\alpha}\\
        = & \frac{14\LL^2 n^\alpha}{1-2\alpha} \left(\frac{2-4\alpha}{(1-\alpha)n^{\frac{\alpha}{2}}L}\right)^{\frac{1-2\alpha}{(1-\alpha)}}\left(\frac{(1-\alpha)n^{\frac{\alpha}{2}}L}{2-4\alpha}\right)^{\frac{1-2\alpha}{1-\alpha}}\left(\sum_{t=1}^T\Norm{\v_t}^2\right)^{1-2\alpha}\\
        \leq & \frac{\alpha}{1-\alpha} \left(\frac{14\LL^2 n^\alpha}{1-2\alpha} \left(\frac{2-4\alpha}{(1-\alpha)n^{\frac{\alpha}{2}}L}\right)^{\frac{1-2\alpha}{1-\alpha}}\right)^{\frac{1-\alpha}{\alpha}}+ \frac{n^{\frac{\alpha}{2}}L}{2}\left(\sum_{t=1}^T\Norm{\v_t}^2\right)^{1-\alpha}
    \end{align*}
    where the last inequality employs Young's inequality, such that $xy \leq \frac{\alpha}{1-\alpha} x^{\frac{1-\alpha}{\alpha}} + \frac{1-2\alpha}{1-\alpha} y^{\frac{1-\alpha}{1-2\alpha}}$ for positive $x, y$.
\end{proof}
As a result, we can ensure that
 \begin{align*}
    \frac{1}{T}\sum_{t=1}^T \E\left[\Norm{\nabla F(\x_t) }\right]  \leq & \frac{1}{T}\E\left[\sum_{t=1}^T  \Norm{\v_t} \right]+  \frac{1}{T}\left[\sum_{t=1}^T\Norm{\nabla f(\x_t) -\v_t} \right] \\
    \leq & \frac{n^{1/4}}{T^{1/2}}\left(  \left(\frac{14\LL^2 }{1-2\alpha} \left(\frac{2-4\alpha}{(1-\alpha)L}\right)^{\frac{1-2\alpha}{1-\alpha}}\right)^{\frac{1-\alpha}{2\alpha}}+\sqrt{2C_3L} + (4C_3)^{\frac{1}{2(1-\alpha)}}\right)\\
=&\mathcal{O}\left(\left(\Delta_F^{\frac{1}{2(1-\alpha)}} + L^{\frac{1}{2\alpha}}\right)\frac{n^{1/4}}{T^{1/2}}\right).
\end{align*}
\section{Proof of Theorem~\ref{T4}}
The analysis is very similar to that of Theorem~\ref{thm:main_2}, and the difference only appears in Lemma~\ref{SAG}. For this new method, we can also prove the same lemma:
\begin{lemma}\label{SVRG+++}
    \begin{align*}
    2\beta \sum_{t=1}^T  \E\left[\eta_t\Norm{\nabla F(\x_{t+1}) -\z_{t+1}}^2\right] \leq 12 \LL^2\sum_{t=1}^T\E\left[\frac{\eta_t^3}{\beta}\Norm{\v_t}^2\right].
    \end{align*}
\end{lemma}
\begin{proof} This time, we have $\z_{t+1} = \nabla f_{i_{t+1}}(\x_{t+1}) - \nabla f_{i_{t+1}}(\x_{\tau}) + \nabla F(\x_{\tau})$. And we can know that:
    \begin{align*}
    \begin{split}
        & \E_{i_{t+1}}\left[\eta_t\Norm{\nabla F(\x_{t+1}) -\z_{t+1}}^2\right] \\
        =&  \E_{i_{t+1}}\left[\eta_t\Norm{\nabla F(\x_{t+1}) -\nabla f_{i_{t+1}}(\x_{t+1}) +  \nabla f_{i_{t+1}}(\x_{\tau}) - \nabla F(\x_{\tau})}^2\right]\\
        \leq &  \E_{i_{t+1}}\left[\eta_t\Norm{\nabla f_{i_{t+1}}(\x_{t+1}) - f_{i_{t+1}}(\x_{\tau})}^2\right] \\
        \leq &   \eta_t L^2 \Norm{\x_{t+1}-\x_{\tau}}^2\\
        \leq &   \eta_t L^2 I \sum_{i=\tau}^{t}\Norm{\x_{i+1}-\x_{i}}^2\\
        \leq &   \eta_t L^2 I \sum_{i=\tau}^{t}\eta_i^2\Norm{\v_{i}}^2  \leq L^2 I \sum_{i=\tau}^{t}\eta_i^3\Norm{\v_{i}}^2.
    \end{split}
    \end{align*}
By summing up, we have
    \begin{align*}
&2\beta \sum_{t=1}^{T}\E\left[\eta_t\Norm{\nabla F(\x_{t+1}) -\z_{t+1}}^2\right]\\
        \leq &2\beta \E\left[\sum_{t=1}^{T} L^2 I \sum_{i=\tau}^{t}\eta_i^3\Norm{\v_{i}}^2 \right]
        \leq
        2\beta\LL^2 I^2\E\left[\sum_{t=1}^{T} \eta_t^3\Norm{\v_t}^2\right] \leq
        2\LL^2 \E\left[\sum_{t=1}^{T} \frac{\eta_t^3}{\beta}\Norm{\v_t}^2\right].
\end{align*}
\end{proof}
The other analysis is exactly the same as that of Theorem~\ref{thm:main_2}.

\newpage
\section*{NeurIPS Paper Checklist}
\begin{enumerate}

\item {\bf Claims}
    \item[] Question: Do the main claims made in the abstract and introduction accurately reflect the paper's contributions and scope?
    \item[] Answer: \answerYes{}
    \item[] Justification: The claims presented in the abstract and introduction accurately represent the contributions and scope of the paper.
    \item[] Guidelines:
    \begin{itemize}
        \item The answer NA means that the abstract and introduction do not include the claims made in the paper.
        \item The abstract and/or introduction should clearly state the claims made, including the contributions made in the paper and important assumptions and limitations. A No or NA answer to this question will not be perceived well by the reviewers. 
        \item The claims made should match theoretical and experimental results, and reflect how much the results can be expected to generalize to other settings. 
        \item It is fine to include aspirational goals as motivation as long as it is clear that these goals are not attained by the paper. 
    \end{itemize}

\item {\bf Limitations}
    \item[] Question: Does the paper discuss the limitations of the work performed by the authors?
    \item[] Answer: \answerYes{}
    \item[] Justification: The theoretical results demonstrated in the paper rely on specific assumptions, which have been clearly stated in the main text.
    \item[] Guidelines: 
    \begin{itemize}
        \item The answer NA means that the paper has no limitation while the answer No means that the paper has limitations, but those are not discussed in the paper. 
        \item The authors are encouraged to create a separate "Limitations" section in their paper.
        \item The paper should point out any strong assumptions and how robust the results are to violations of these assumptions (e.g., independence assumptions, noiseless settings, model well-specification, asymptotic approximations only holding locally). The authors should reflect on how these assumptions might be violated in practice and what the implications would be.
        \item The authors should reflect on the scope of the claims made, e.g., if the approach was only tested on a few datasets or with a few runs. In general, empirical results often depend on implicit assumptions, which should be articulated.
        \item The authors should reflect on the factors that influence the performance of the approach. For example, a facial recognition algorithm may perform poorly when image resolution is low or images are taken in low lighting. Or a speech-to-text system might not be used reliably to provide closed captions for online lectures because it fails to handle technical jargon.
        \item The authors should discuss the computational efficiency of the proposed algorithms and how they scale with dataset size.
        \item If applicable, the authors should discuss possible limitations of their approach to address problems of privacy and fairness.
        \item While the authors might fear that complete honesty about limitations might be used by reviewers as grounds for rejection, a worse outcome might be that reviewers discover limitations that aren't acknowledged in the paper. The authors should use their best judgment and recognize that individual actions in favor of transparency play an important role in developing norms that preserve the integrity of the community. Reviewers will be specifically instructed to not penalize honesty concerning limitations.
    \end{itemize}

\item {\bf Theory Assumptions and Proofs}
    \item[] Question: For each theoretical result, does the paper provide the full set of assumptions and a complete (and correct) proof?
    \item[] Answer: \answerYes{}
    \item[] Justification: The paper provides assumptions and proofs for each theoretical result.
    \item[] Guidelines:
    \begin{itemize}
        \item The answer NA means that the paper does not include theoretical results. 
        \item All the theorems, formulas, and proofs in the paper should be numbered and cross-referenced.
        \item All assumptions should be clearly stated or referenced in the statement of any theorems.
        \item The proofs can either appear in the main paper or the supplemental material, but if they appear in the supplemental material, the authors are encouraged to provide a short proof sketch to provide intuition. 
        \item Inversely, any informal proof provided in the core of the paper should be complemented by formal proofs provided in appendix or supplemental material.
        \item Theorems and Lemmas that the proof relies upon should be properly referenced. 
    \end{itemize}

    \item {\bf Experimental Result Reproducibility}
    \item[] Question: Does the paper fully disclose all the information needed to reproduce the main experimental results of the paper to the extent that it affects the main claims and/or conclusions of the paper (regardless of whether the code and data are provided or not)?
    \item[] Answer: \answerYes{}
    \item[] Justification: The paper discloses the information necessary to reproduce the main experimental results.
    \item[] Guidelines:
    \begin{itemize}
        \item The answer NA means that the paper does not include experiments.
        \item If the paper includes experiments, a No answer to this question will not be perceived well by the reviewers: Making the paper reproducible is important, regardless of whether the code and data are provided or not.
        \item If the contribution is a dataset and/or model, the authors should describe the steps taken to make their results reproducible or verifiable. 
        \item Depending on the contribution, reproducibility can be accomplished in various ways. For example, if the contribution is a novel architecture, describing the architecture fully might suffice, or if the contribution is a specific model and empirical evaluation, it may be necessary to either make it possible for others to replicate the model with the same dataset, or provide access to the model. In general. releasing code and data is often one good way to accomplish this, but reproducibility can also be provided via detailed instructions for how to replicate the results, access to a hosted model (e.g., in the case of a large language model), releasing of a model checkpoint, or other means that are appropriate to the research performed.
        \item While NeurIPS does not require releasing code, the conference does require all submissions to provide some reasonable avenue for reproducibility, which may depend on the nature of the contribution. For example
        \begin{enumerate}
            \item If the contribution is primarily a new algorithm, the paper should make it clear how to reproduce that algorithm.
            \item If the contribution is primarily a new model architecture, the paper should describe the architecture clearly and fully.
            \item If the contribution is a new model (e.g., a large language model), then there should either be a way to access this model for reproducing the results or a way to reproduce the model (e.g., with an open-source dataset or instructions for how to construct the dataset).
            \item We recognize that reproducibility may be tricky in some cases, in which case authors are welcome to describe the particular way they provide for reproducibility. In the case of closed-source models, it may be that access to the model is limited in some way (e.g., to registered users), but it should be possible for other researchers to have some path to reproducing or verifying the results.
        \end{enumerate}
    \end{itemize}

\item {\bf Open access to data and code}
    \item[] Question: Does the paper provide open access to the data and code, with sufficient instructions to faithfully reproduce the main experimental results, as described in supplemental material?
    \item[] Answer: \answerNo{} 
    \item[] Justification:  Due to privacy concerns and ongoing research, we do not include the code.
    \item[] Guidelines:
    \begin{itemize}
        \item The answer NA means that paper does not include experiments requiring code.
        \item Please see the NeurIPS code and data submission guidelines (\url{https://nips.cc/public/guides/CodeSubmissionPolicy}) for more details.
        \item While we encourage the release of code and data, we understand that this might not be possible, so “No” is an acceptable answer. Papers cannot be rejected simply for not including code, unless this is central to the contribution (e.g., for a new open-source benchmark).
        \item The instructions should contain the exact command and environment needed to run to reproduce the results. See the NeurIPS code and data submission guidelines (\url{https://nips.cc/public/guides/CodeSubmissionPolicy}) for more details.
        \item The authors should provide instructions on data access and preparation, including how to access the raw data, preprocessed data, intermediate data, and generated data, etc.
        \item The authors should provide scripts to reproduce all experimental results for the new proposed method and baselines. If only a subset of experiments are reproducible, they should state which ones are omitted from the script and why.
        \item At submission time, to preserve anonymity, the authors should release anonymized versions (if applicable).
        \item Providing as much information as possible in supplemental material (appended to the paper) is recommended, but including URLs to data and code is permitted.
    \end{itemize}

\item {\bf Experimental Setting/Details}
    \item[] Question: Does the paper specify all the training and test details (e.g., data splits, hyperparameters, how they were chosen, type of optimizer, etc.) necessary to understand the results?
    \item[] Answer: \answerYes{} 
    \item[] Justification: The paper describes the training and testing details.
    \item[] Guidelines:
    \begin{itemize}
        \item The answer NA means that the paper does not include experiments.
        \item The experimental setting should be presented in the core of the paper to a level of detail that is necessary to appreciate the results and make sense of them.
        \item The full details can be provided either with the code, in appendix, or as supplemental material.
    \end{itemize}

\item {\bf Experiment Statistical Significance}
    \item[] Question: Does the paper report error bars suitably and correctly defined or other appropriate information about the statistical significance of the experiments?
    \item[] Answer: \answerYes{} 
    \item[] Justification: The paper reports error bars.
    \item[] Guidelines:
    \begin{itemize}
        \item The answer NA means that the paper does not include experiments.
        \item The authors should answer "Yes" if the results are accompanied by error bars, confidence intervals, or statistical significance tests, at least for the experiments that support the main claims of the paper.
        \item The factors of variability that the error bars are capturing should be clearly stated (for example, train/test split, initialization, random drawing of some parameter, or overall run with given experimental conditions).
        \item The method for calculating the error bars should be explained (closed form formula, call to a library function, bootstrap, etc.)
        \item The assumptions made should be given (e.g., Normally distributed errors).
        \item It should be clear whether the error bar is the standard deviation or the standard error of the mean.
        \item It is OK to report 1-sigma error bars, but one should state it. The authors should preferably report a 2-sigma error bar than state that they have a 96\% CI, if the hypothesis of Normality of errors is not verified.
        \item For asymmetric distributions, the authors should be careful not to show in tables or figures symmetric error bars that would yield results that are out of range (e.g. negative error rates).
        \item If error bars are reported in tables or plots, The authors should explain in the text how they were calculated and reference the corresponding figures or tables in the text.
    \end{itemize}

\item {\bf Experiments Compute Resources}
    \item[] Question: For each experiment, does the paper provide sufficient information on the computer resources (type of compute workers, memory, time of execution) needed to reproduce the experiments?
    \item[] Answer: \answerYes{} 
    \item[] Justification: We have provided the relevant information.
    \item[] Guidelines:
    \begin{itemize}
        \item The answer NA means that the paper does not include experiments.
        \item The paper should indicate the type of compute workers CPU or GPU, internal cluster, or cloud provider, including relevant memory and storage.
        \item The paper should provide the amount of compute required for each of the individual experimental runs as well as estimate the total compute. 
        \item The paper should disclose whether the full research project required more compute than the experiments reported in the paper (e.g., preliminary or failed experiments that didn't make it into the paper). 
    \end{itemize}
    
\item {\bf Code Of Ethics}
    \item[] Question: Does the research conducted in the paper conform, in every respect, with the NeurIPS Code of Ethics \url{https://neurips.cc/public/EthicsGuidelines}?
    \item[] Answer: \answerYes{} 
    \item[] Justification: The research conducted in the paper conforms with the NeurIPS Code of Ethics.
    \item[] Guidelines:
    \begin{itemize}
        \item The answer NA means that the authors have not reviewed the NeurIPS Code of Ethics.
        \item If the authors answer No, they should explain the special circumstances that require a deviation from the Code of Ethics.
        \item The authors should make sure to preserve anonymity (e.g., if there is a special consideration due to laws or regulations in their jurisdiction).
    \end{itemize}

\item {\bf Broader Impacts}
    \item[] Question: Does the paper discuss both potential positive societal impacts and negative societal impacts of the work performed?
    \item[] Answer: \answerNA{} 
    \item[] Justification: This is primarily a theoretical paper with no potential negative social impact. 
    \item[] Guidelines:
    \begin{itemize}
        \item The answer NA means that there is no societal impact of the work performed.
        \item If the authors answer NA or No, they should explain why their work has no societal impact or why the paper does not address societal impact.
        \item Examples of negative societal impacts include potential malicious or unintended uses (e.g., disinformation, generating fake profiles, surveillance), fairness considerations (e.g., deployment of technologies that could make decisions that unfairly impact specific groups), privacy considerations, and security considerations.
        \item The conference expects that many papers will be foundational research and not tied to particular applications, let alone deployments. However, if there is a direct path to any negative applications, the authors should point it out. For example, it is legitimate to point out that an improvement in the quality of generative models could be used to generate deepfakes for disinformation. On the other hand, it is not needed to point out that a generic algorithm for optimizing neural networks could enable people to train models that generate Deepfakes faster.
        \item The authors should consider possible harms that could arise when the technology is being used as intended and functioning correctly, harms that could arise when the technology is being used as intended but gives incorrect results, and harms following from (intentional or unintentional) misuse of the technology.
        \item If there are negative societal impacts, the authors could also discuss possible mitigation strategies (e.g., gated release of models, providing defenses in addition to attacks, mechanisms for monitoring misuse, mechanisms to monitor how a system learns from feedback over time, improving the efficiency and accessibility of ML).
    \end{itemize}
    
\item {\bf Safeguards}
    \item[] Question: Does the paper describe safeguards that have been put in place for responsible release of data or models that have a high risk for misuse (e.g., pretrained language models, image generators, or scraped datasets)?
    \item[] Answer: \answerNA{} 
    \item[] Justification: The paper poses no such risks.
    \item[] Guidelines:
    \begin{itemize}
        \item The answer NA means that the paper poses no such risks.
        \item Released models that have a high risk for misuse or dual-use should be released with necessary safeguards to allow for controlled use of the model, for example by requiring that users adhere to usage guidelines or restrictions to access the model or implementing safety filters. 
        \item Datasets that have been scraped from the Internet could pose safety risks. The authors should describe how they avoided releasing unsafe images.
        \item We recognize that providing effective safeguards is challenging, and many papers do not require this, but we encourage authors to take this into account and make a best faith effort.
    \end{itemize}

\item {\bf Licenses for existing assets}
    \item[] Question: Are the creators or original owners of assets (e.g., code, data, models), used in the paper, properly credited and are the license and terms of use explicitly mentioned and properly respected?
    \item[] Answer: \answerYes{} 
    \item[] Justification: The creators or original owners of assets used in the paper are properly credited and the license and terms of use explicitly are properly respected.
    \item[] Guidelines:
    \begin{itemize}
        \item The answer NA means that the paper does not use existing assets.
        \item The authors should cite the original paper that produced the code package or dataset.
        \item The authors should state which version of the asset is used and, if possible, include a URL.
        \item The name of the license (e.g., CC-BY 4.0) should be included for each asset.
        \item For scraped data from a particular source (e.g., website), the copyright and terms of service of that source should be provided.
        \item If assets are released, the license, copyright information, and terms of use in the package should be provided. For popular datasets, \url{paperswithcode.com/datasets} has curated licenses for some datasets. Their licensing guide can help determine the license of a dataset.
        \item For existing datasets that are re-packaged, both the original license and the license of the derived asset (if it has changed) should be provided.
        \item If this information is not available online, the authors are encouraged to reach out to the asset's creators.
    \end{itemize}

\item {\bf New Assets}
    \item[] Question: Are new assets introduced in the paper well documented and is the documentation provided alongside the assets?
    \item[] Answer: \answerNA{} 
    \item[] Justification: The paper does not release new assets.
    \item[] Guidelines:
    \begin{itemize}
        \item The answer NA means that the paper does not release new assets.
        \item Researchers should communicate the details of the dataset/code/model as part of their submissions via structured templates. This includes details about training, license, limitations, etc. 
        \item The paper should discuss whether and how consent was obtained from people whose asset is used.
        \item At submission time, remember to anonymize your assets (if applicable). You can either create an anonymized URL or include an anonymized zip file.
    \end{itemize}

\item {\bf Crowdsourcing and Research with Human Subjects}
    \item[] Question: For crowdsourcing experiments and research with human subjects, does the paper include the full text of instructions given to participants and screenshots, if applicable, as well as details about compensation (if any)? 
    \item[] Answer: \answerNA{} 
    \item[] Justification: The paper does not involve crowdsourcing nor research with human subjects.
    \item[] Guidelines:
    \begin{itemize}
        \item The answer NA means that the paper does not involve crowdsourcing nor research with human subjects.
        \item Including this information in the supplemental material is fine, but if the main contribution of the paper involves human subjects, then as much detail as possible should be included in the main paper. 
        \item According to the NeurIPS Code of Ethics, workers involved in data collection, curation, or other labor should be paid at least the minimum wage in the country of the data collector. 
    \end{itemize}

\item {\bf Institutional Review Board (IRB) Approvals or Equivalent for Research with Human Subjects}
    \item[] Question: Does the paper describe potential risks incurred by study participants, whether such risks were disclosed to the subjects, and whether Institutional Review Board (IRB) approvals (or an equivalent approval/review based on the requirements of your country or institution) were obtained?
    \item[] Answer: \answerNA{} 
    \item[] Justification: the paper does not involve crowdsourcing nor research with human subjects.
    \item[] Guidelines:
    \begin{itemize}
        \item The answer NA means that the paper does not involve crowdsourcing nor research with human subjects.
        \item Depending on the country in which research is conducted, IRB approval (or equivalent) may be required for any human subjects research. If you obtained IRB approval, you should clearly state this in the paper. 
        \item We recognize that the procedures for this may vary significantly between institutions and locations, and we expect authors to adhere to the NeurIPS Code of Ethics and the guidelines for their institution. 
        \item For initial submissions, do not include any information that would break anonymity (if applicable), such as the institution conducting the review.
    \end{itemize}

\end{enumerate}

\end{document}